\documentclass[twoside]{report}
\usepackage{amssymb,lj-igpl}
\usepackage{color}
\usepackage{cmll}
\usepackage{stmaryrd}
\usepackage{MnSymbol}
\usepackage{subfig}
\usepackage{amsmath}  
\newtheorem{definition}{Definition}
\newtheorem{example}[definition]{Example}

\newtheorem{theorem}[definition]{Theorem}

\newtheorem{lemma}[definition]{Lemma}

\newtheorem{proposition}[definition]{Proposition}

\newtheorem{corollary}[definition]{Corollary}

\makeatletter
\renewcommand{\@begintheorem}[2]{ % not in italics
\trivlist\item[\hskip\labelsep{\bf #1\ #2}]}
\renewcommand{\@opargbegintheorem}[3]{\trivlist
\item[\hskip \labelsep{\bf #1\ #2\ (#3)}]}
\makeatother
\newtheorem{proof}{Proof}
\usepackage{doc}
\usepackage{proof}
\usepackage{amsfonts,mathrsfs,amssymb,amsmath}
\newcommand\blfootnote[1]{%
  \begingroup
  \renewcommand\thefootnote{}\footnote{#1}%
  \addtocounter{footnote}{-1}%
  \endgroup
}

 \usepackage{rotating}
\usepackage{hyperref}

\def\d{\Delta}
\newcommand{\GSD}{\Gamma\Rightarrow\Delta}
\def\g{\Gamma}

\def\s{\sigma}

\def\to{\supset}
\def\tto{\supset\subset}

\newcommand{\infrule}[1]{\scriptstyle\it{#1}}

\newcommand{\riota}{\rotatebox[origin=c]{180}{$\iota$}}

\def\<{\langle}
\def\>{\rangle}
\def\to{\supset}
\def\To{\Rightarrow}

\def\ex{\exists}
\def\fa{\forall}
\def\l{\lambda}
\def\F{\mathcal{F}}
\def\M{\mathcal{M}}
\def\MM{\mathscr{M}}
\def\W{\mathcal{W}}

\def\D{\mathcal{D}}
\def\C{\mathcal{C}}

\def\R{\mathcal{R}}

\def\B{\mathcal{B}}
\def\V{\mathcal{V}}
\def\L{\mathcal{L}}

\def\tr{\triangleright}

\def\RR{\mathscr{R}}
\Title{Labelled calculi for quantified modal logics with  definite descriptions}
\ShortAuthor{ E. Orlandelli}
\LongAuthor{
\author{Eugenio Orlandelli}
\address{Department of Philosophy and Communication Studies, University of Bologna, via Zamboni 38, Bologna, Italy.\\ Email: eugenio.orlandelli@unibo.it}
}

\Received{Day Month year.}
\begin{document}
\begin{paper}

\begin{abstract}
We introduce  labelled sequent calculi for quantified modal logics with definite descriptions. We prove that these calculi have the good structural properties of {\bf G3}-style calculi. In particular, all rules are height-preserving invertible, weakening and contraction are height-preserving admissible and cut is admissible. Finally, we show that each calculus gives a proof-theoretic characterization of validity in the corresponding class of models.\blfootnote{I am grateful to Sara Negri for valuable discussion of the ideas and results  presented in this paper.}  %In particular, the completeness proof constructs a formal derivation for derivable sequents and a countermodel for non-derivable ones, and  gives a semantic proof of the admissibility of cut. 
\end{abstract}

%--------   	INTRODUCTION
\section{Introduction}
%\cite{I18,G13,G11}

The proof-theoretic study of propositional modal logics is now a well-developed subject thanks to the introduction of generalizations of Gentzen-style sequent calculi. In particular, we have   internal calculi --  e.g., hypersequents \cite{CRW} and nested sequents \cite{B09} -- whose sequents are interpretable in the modal language, and we have external calculi -- e.g., display calculi \cite{CRW} and labelled sequent calculi \cite{NP11} -- whose sequents are not interpretable in the basic modal language.  Nevertheless,  with the only exception of labelled calculi \cite{CO16,NO19,NP11,OC18,CO18}, the proof-theoretic study of  QMLs has remained  rather underdeveloped, see \cite{R16} for some considerations on hypersequents and display calculi for QMLs. One interesting  problem that is still open is that of presenting a satisfactory approach to the  structural proof theory for quantified modal logics (QMLs) with definite descriptions: the only cut-free calculi are the Gentzen-style calculi for QMLs with definite description \emph{\`a la} Garson \cite{G13} that have been presented in \cite{I18}. 

Starting from our work in \cite{OC18}, we will introduce labelled calculi for the  QMLs with descriptions and  $\lambda$-abstraction that are studied by Fitting and Mendelsohn  \cite{FM98}. We will show that these calculi have  good structural properties -- all rules are height-preserving invertible, weakening and contraction are height-preserving admissible, and cut is (syntactically) admissible -- and characterize validity in the appropriate semantic classes. In so doing we  solve a problem  left open in \cite{I18} where we read:

\begin{quote}
[Fitting and Mendelsohn's one] is probably the most subtle theory of definite descriptions [...] As such it it certainly deserves attention but it is difficult to provide a suitable sequent formalization of it. [p. 388]
\end{quote}%One notable exception is \cite[Chap.~12.1]{NP11}, where labelled sequent calculi for QMLs are introduced. More specifically, the labelled calculi for the propositional modal logics in the cube of normal modalities -- i.e. the minimal normal modal logic {\bf K} and  its extensions with axioms $D,T,4,5$ -- are extended with quantifiers based on varying, increasing, decreasing and constant domains. One interesting aspect of QMLs that has not been considered in \cite{NP11} is the study of logics based on a language containing the identity predicate and non-rigid as well as non-denoting terms, see \cite{F99,FM98}. We introduce labelled calculi for these logics and we study their structural properties.
The rest of this introduction gives a quick introduction to Fitting and Mendelsohn's QMLs with definite descriptions and  explains why labelled calculi  are the ideal formalism to study their structural proof theory.

As it is convincingly argued in \cite{FM98}, the need for non-rigid and non-denoting terms  originates from problems already touched upon in the classical works of Frege \cite{F92} and Russell \cite{R05}. First, as Frege noticed, even if both `the morning star' and `the evening star' denote Venus and even if the ancient knew that objects are self-identical, the Babylonians did not know that `the morning star is identical with the evening star'. Despite this, if we treat definite descriptions as genuine terms,  in standard QMLs we can prove that the Babylonians knew it because terms are rigid designators. Moreover, Russell showed that the sentence `The present king of France is  not bald' is ambiguous since it might either mean that the sentence `the present king of France is bald' is false, or that  the present king of France is  such that he is non-bald. Given that the expression `The present king of France' does not actually denote anyone, the first reading is, in fact, true and the second false. If we exclude non-denoting terms, we cannot account for these two readings of our sentence (unless we explain away the term expressing these definite descriptions).

As the two examples above show, if definite descriptions are taken as genuine terms,  we must extend  the language of  QMLs with non-rigid and non-denoting terms, but this extension is not trivial \cite{FM98}. The  problem, roughly, is that if $t$ is a non-rigid or  non-denoting term, the formal sentences $\Diamond Pt$  and $\neg Pt$  become ambiguous.  when  it  is evaluated in a possible world $w$ of some model, the sentence  $\Diamond Pt$  might either mean that  there is a world $u$ that is accessible from $w$ and such that the formal sentence $Pt$ is true therein, or it might mean that the object denoted by $t$ in $w$ satisfies the unary predicate $P$ is some world $v$ that is accessible from $w$. Analogously,  $\neg Pt$ might either mean that  it is false that (in $w$) there is one and only object that is denoted by $t$ and that  satisfies $P$, or it might mean that the one and only object denoted by $t$ (in $w$) does not satisfy $P$. For rigid and always denoting  terms the two readings are equivalent. For non-rigid and non-denoting  terms neither reading entails the other, and, therefore, we need some scoping mechanism to disambiguate the formulas $\Diamond Pt$  and $\neg Pt$. The solution adopted in \cite{FM98}  is that of extending the language with  the operator of predicate abstraction $\l$. The two readings of $\circ Pt$ (for $\circ\in\{\Diamond,\neg\}$) can thus be expressed, respectively, by the (semantically independent) formal sentences:
$$%\begin{equation}\label{lambda}
\l x(\circ Px).t\qquad\text{and}\qquad \circ(\l xPx.t)
$$%\end{equation}
All in all, Russell's \cite{R05} (and Smullyan's \cite{S48})  proposal of explaining away definite descriptions by means  of  quantification and identity originates from the need to have a scoping mechanism for the formal representation of non-rigid and non-denoting terms. By using the  machinery of  $\lambda$-abstraction we have a scoping mechanism for terms and, therefore, we don't need anymore to explain them away. One essential  feature of this approach is that non-rigid and non-denoting terms, such as definite descriptions, can occur in formulas only when they are applied by the operator $\l$ and not as one of the \emph{relata} of an atomic formula. This is needed because, otherwise, we would have problem in interpreting $\Diamond Pt$ (for $t$ non-rigid) and $\neg Pt$ (for $t$ non-denoting).

It is well-known that labelled calculi allows to give well-behaved sequent calculi for all first-order semantically definable propositional modal logic. The key idea is that of extending the language of sequent calculus in order to internalize relational semantics into the syntax: we add world labels (representing worlds) and relational atoms (representing the accessibility relation), and we replace modal formulas with labelled modal formulas (representing truth in a world). This allows to give well-behaved rules for the modalities (that are just like rules for restricted first-order  quantifiers). Moreover, thanks to the presence of relational atoms, it allows to use the method of axioms as rules  \cite{N03,NP98} to transform the first-order semantic conditions that define interesting modal logics into rules of the calculus.  This can be done directly for geometric semantic conditions (i.e., formulas of shape $\forall \vec{x}(A\to B)$, where neither $A$ nor $B$ contains $\forall$ and $\to$) and indirectly  for non-geometric ones (via the method of geometrisation of arbitrary first-order formulas \cite{DN15}).

As we will show, the strategy of internalizing the semantics works equally well for QMLs with definite descriptions. In particular, in order to internalize the semantics presented in \cite{FM98} we will need to add to the labelled language also denotation formulas of shape $D(t,x,w)$ that, when $t$ is a definite description, express the non-trivial fact that $t$ denotes one object in the world $w$. In this way we can easily define an external calculus for the QMLs with definite descriptions presented in \cite{FM98}.  Even if it is possible to define well-behaved internal calculi for some QMLs without definite descriptions \cite{R16}, e.g., by applying the embeddings given in \cite{G18,LR18}. We believe that labelled calculi provides the best tool for the logics we are considering because  it seems hard to define a calculus for them without using denotation formulas or some other extension of the language that cannot be interpreted in the language of modal logic. In a nutshell, the problem is that something like denotation formulas are needed to cope with definite descriptions and   the only way to interpret a denotation formula $D(t,x,w)$ in the object language (or to do without something like denotation formulas as  in \cite{I18,I19})  is via  an identity atom of shape $t=x$. But, if the modal language allows for formulas of shape $t=x$ with $t$ a definite description, then  $\lambda$-abstraction looses its role of  scoping mechanism and we run into problems with substitutivity of identicals, see \cite[Chapter 10.1]{FM98}, and with cut-elimination for identity atoms  \cite[Section 5]{I18}.%Observe that, contrary to relational atoms, description formulas need not be atomic since  definite descriptions contain formulas of arbitrary complexity. Hence, as for forcing formulas in labelled calculi for non-normal modal logics, descriptions formulas behave more like labelled formulas than like relational atoms: they might be principal in initial sequents and in rules (cut included)  on both sides of the sequents. 
%We will extend the labelled sequent calculi for QMLs presented in \cite{NP11} to handle non-rigid and non-denoting terms based on the predicate abstraction operator $\l$. We will study the structural properties of these extensions, and we will show that, as  for the   calculi in \cite{NP11}, all the rules of inference are height-preserving invertible, the structural rules of weakening and contraction are height-preserving admissible, and  the rule of cut is admissible. Finally, we will prove that each calculus considered characterizes validity in the appropriate class of models.

The paper is organized  as follows. Section \ref{noid} sketches the labelled calculi for QMLs presented in \cite{NP11}.  In particular,  the language and semantics of standard QMLs   are introduced in Section \ref{sect:syntax}, and  labelled calculi for these logics   are outlined in Section \ref{lsc}. In Section \ref{sec:definite}, we introduce QMLs with identity and definite descriptions.  We briefly compare    different approaches to descriptions (Section \ref{secintrdd}) and  we present  the syntax and the semantics of the QMLs with definite descriptions presented in \cite{FM98} (Section \ref{sec:descsem}).  Then, in Section  \ref{sec:desclab}, we introduce labelled calculi for these logics. Section \ref{sec:structural} shows that  these calculi have the good structural properties that are distinctive of {\bf G3}-style calculi, and Section  \ref{sec:characterization} shows that they are sound and complete with respect to the appropriate classes of quantified modal frames. We conclude in Section \ref{conclusion} by showing how the present approach % and by arguing that it 
can be extended to cover the quantified extensions of all first-order definable propositional modal logics and how it can simulate some other  approache to definite descriptions.

%----------   QUANTIFIED MODAL LOGIC
%------------------------------------------------------------------------------
\section{Quantified Modal Logics}\label{noid}
In this section, we present QMLs based on a varying domain semantics  defined over a signature not containing  functions of any arity  nor the identity symbol, and we present labelled calculi for these logics. Apart from some minor  adjustment,  the semantics is as in \cite[Chap.~4.7]{FM98}, and the calculi are as in \cite[Chap.~12.1]{NP11}.  This section is needed to make the paper self-contained and it might be skipped by readers already familiar with QMLs and labelled calculi.
\subsection{Syntax and Semantics}
\label{sect:syntax}

Let $\mathcal{S}$ be a signature containing, for every $n\in\mathbb{N}$, an at most denumerable set $REL^\mathcal{S}$ of $n$-ary predicate letters $P_1^n,P_2^n,\dots$, and let $VAR$ be an infinite set of variables $x_1,x_2,\dots$. The language $\mathscr{L}$  is given by the grammar:
\begin{equation}\tag{$\mathcal{L}$}\label{L}
A::=P^nx_1,\dots,x_n\;|\; \bot\;|\;A\wedge A\;|\;A\lor A\;|\;A\to A\;|\;\forall xA\;|\;\exists xA\;|\;\Box A\;|\;\Diamond A
\end{equation}
where $P_n\in REL^\mathcal{S}$ and $x,x_1,\dots,x_n\in VAR$. We  use the following metavariables:
\begin{itemize}
\item $P,Q,R$ for predicate letters;
\item $x,y,z$ for variables;
\item $p,q,r$ for atomic formulas;
\item $A,B,C$ for formulas.
\end{itemize}
We follow the standard conventions for parentheses. The formulas $\top,\,\neg A$ and $ A\tto B$ are defined as expected. The notions of \emph{free} and \emph{bound occurrences} of a variable in a formula are the usual ones. Given a formula $A$, we use $A[y/x]$ to denote the formula  obtained by replacing each free occurrence of $x$ in $A$ with an occurrence of $y$, provided that $y$ is free for $x$ in $A$ -- i.e.,  no  new occurrence of $y$ is bound by a quantifier. %The \emph{height} of a formula is the number of nodes in the longest branch in  its generation tree.

A \emph{model} (over the signature $\mathcal{S}$) is a tuple:
$$%\begin{equation}\label{model}
\mathcal{M}=\<\W,\R,\D,\V\>%\qquad\text{ where}
 $$%\end{equation}
where
\begin{itemize}
\item $\W\neq\varnothing$ is a nonempty set of (possible) \emph{worlds} (to be denoted by $w,v,u\dots$);
\item $\R\subseteq \W\times \W$ is a binary \emph{accessibility relation} between worlds;
\item $\D: \W\longrightarrow 2^D$ is a function mapping each world to a possibly empty set of objects $D_w$ (its \emph{domain}),  where $D_\W=\bigcup_{w\in \W} D_w$ is nonempty and disjoint from $\W$;
\item $\V:\mathcal{S}\times \W\longrightarrow 2^{(D_\W)^n}$ is a \emph{valuation} function mapping, at each world $w$, each $n$-ary predicate $P(\in\mathcal{S})$ to a subset of $(D_\W)^n$.
\end{itemize}

A \emph{frame} $\F$ is a triple $\<\W,\R, \D\>$ (i.e. it is a model without valuation),  and a model $\M$ is  \emph{based on a frame} $\F$ if $\M=\<\F,\V\>$. We will say that a frame $\<\W,\R,\D\>$ has:
\begin{itemize}
\item \emph{Increasing domain} if $\forall w,v\in\W$,\, $w \R v$ implies $D_w\subseteq D_v$;
\item \emph{decreasing domain} if $\forall w,v\in\W$,\, $w \R v$ implies $D_w\supseteq D_v$;
\item \emph{Constant  domain} if  $\forall w,v\in\W$,\, $w \R v$ implies $D_w= D_v$.
\end{itemize}

Given a model $\MM=\<\W,\R,\D,\V\>$, an \emph{assignment} (over $\MM$) is a function \mbox{$\sigma: VAR\longrightarrow D_\W$} mapping each variable $x$ to an element of the union of the domains of the model. Moreover, for $o\in D$, $\sigma^{x\tr o}$ denotes the assignment behaving like $\sigma$ save for  $x$ that is mapped to the object $o$. 

%We have now all the elements to define the notion of \emph{satisfaction}.

\begin{definition}[Satisfaction]\label{sat}
Given a model $\MM$, an assignment $\sigma$ over it, and a world $w$ of that model, we define the notion of \emph{satisfaction} of an $\mathscr{L}$-formula $A$ as follows:

\begin{tabular}{lll}\noalign{\smallskip}
$\sigma\models^\MM_w Px_1,\dots x_n$&iff& $\<\sigma(x_1),\dots,\sigma(x_n)\>\in \V(P,w)$\\\noalign{\smallskip}
$\sigma\nmodels^\MM_w \bot$\\\noalign{\smallskip}
$\sigma\models^\MM_w B\wedge C$&iff& $\sigma\models^\MM_w B$ and $\sigma\models^\MM_w C$\\\noalign{\smallskip}
$\sigma\models^\MM_w B\lor C$&iff& $\sigma\models^\MM_w B$ or $\sigma\models^\MM_w C$\\\noalign{\smallskip}
$\sigma\models^\MM_w B\to C$&iff& $\sigma\nmodels^\MM_w B$ or $\sigma\models^\MM_w C$\\\noalign{\smallskip}
$\sigma\models^\MM_w \forall xB$&iff& for all $o\in D_w,\, \sigma^{x\tr o}\models^\MM_w B$\\\noalign{\smallskip}
$\sigma\models^\MM_w \exists xB$&iff& for some $o\in D_w,\, \sigma^{x\tr o}\models^\MM_w B$\\\noalign{\smallskip}
$\sigma\models^\MM_w \Box B$&iff& for all $v\in\W,\; w\R v$ implies $\sigma\models^\MM_v B$\\\noalign{\smallskip}
$\sigma\models^\MM_w \Diamond B$&iff& for some $v\in\W,\; w\R v$ and $\sigma\models^\MM_v B$\\ 
\end{tabular}
\end{definition}

The notions of \emph{truth in a world  $w$ of a  model}  ($\models_w^\MM A$), \emph{truth in a model} ($\models^\MM A$), and \emph{validity in a (class of) frames}  ($\F(\in\C) \models A$) are defined as usual. 

 As it is well known,  some notable formulas are valid in classes of frames  defined by properties of the accessibility relation and/or of the domains. In particular, Table \ref{prop} presents some well-known (geometric) propositional correspondence results as well as correspondence results for increasing, decreasing and constant domain frames. By an \emph{$\mathscr{L}$-logic}  $\mathbf{Q.L}$ we mean the set of all $\mathscr{L}$-formulas that are valid in a class of frames. We use standard names for  $\mathscr{L}$-logics -- e.g., {\bf Q.K} stands for the set of $\mathscr{L}$-formulas valid in the class of all frames, and {\bf Q.S4${}\oplus$CBF/BF/UI} stands for the set of $\mathscr{L}$-formulas valid in the class of all  reflexive and transitive frames with increasing/decreasing/constant domain. We say  that \emph{$\MM$ is a model for {\bf Q.L}} whenever $\MM$ is based on a frame in the class that defines {\bf Q.L}.

\begin{table}\caption{Modal axioms and corresponding semantic properties}\label{prop}
\begin{center}\begin{tabular}{ll}
\hline\hline\noalign{\smallskip}
$T:= \Box A\to A$&reflexivity:=$ \forall w\in \W(w\R w)$\\\noalign{\smallskip}
$D:=\Box A\to\Diamond A$&seriality:=$\forall w\in\W\exists u\in\W(w\R u)$\\\noalign{\smallskip}
$4:=\Box A\to\Box\Box A$&transitivity:=$\forall w,v,u\in\W(w\R v\wedge v\R u\to w\R u)$\\\noalign{\smallskip}
$5:=\Diamond A\to\Box\Diamond A$&Euclideaness:=$\forall w,v,u\in\W(w\R v\wedge w\R u\to v\R u)$\\\noalign{\smallskip}
%$NE:=\forall xA\to\exists xA$&nonempty domains:=$ \forall w\in\W\exists a\in D(a\in D_w)$\\\noalign{\smallskip}
%$UI:=\forall x A\to A[y/x]\forall w\in\W\forall a\in D(a\in D_w)$
$CBF:=\Box\forall xA\to\forall x\Box A$&increasing domain%:=$\forall w,v\in \W\forall a\in D(a\in D_w\wedge w\R v\to a\in D_v)$
\\\noalign{\smallskip}
$BF:=\forall x\Box A\to\Box\forall xA\quad$&decreasing domain%:=$\forall w,v\in \W\forall a\in D(a\in D_v\wedge w\R v\to a\in D_w)$
\\\noalign{\smallskip}
$UI:=\forall xA\to A[y/x]$&constant domain
\\\noalign{\smallskip}\hline\hline
\end{tabular}\end{center}
\end{table}
%------------------------------------------------------------------------------
\subsection{Labelled Sequent Calculi}
\label{lsc}
Labelled sequent calculi for $\mathscr{L}$-logics  have been introduced in \cite[Chapter~12.1]{NP11}.  These calculi are based on  extending the modal  language in order to internalize the semantics of QMLs. First of all, we introduce a  set $LAB$ of fresh variables, called \emph{labels}. Labels will be denoted by $w,v,u,\dots$ and will be used to represent worlds. Then, we extend the set of formulas by adding atomic formulas of shape $x\in w$ --  expressing that (the object assigned to) $x$ is in the domain of quantification of (the world represented by) $w$ -- and of shape $w\RR v$ --  expressing that  $v$ is accessible from   $w$. Lastly, we replace each $\mathscr{L}$-formula $A$ with the labelled formula $w:A$ --  expressing that $A$ holds at  $w$. A \emph{labelled sequent} is an expression:
$$
{}\GSD
$$
where $\g$ is a multiset  composed of labelled formulas and of  atomic formulas of shape $x\in w$ or $w\RR v$, and where $\d$ is a multiset  of labelled formulas. Given a formula $E$ of this extended language, $E[w/v]$ is the formula obtained by substituting each occurrence of $v$ in $E$ with an occurrence of $w$. Substitution  of variables is extended to formulas of the extended language as expected, and both kinds of substitution are extended to sequents by applying them componentwise.

The rules of the calculus {\bf G3Q.K}, for the minimal $\mathscr{L}$-logic {\bf Q.K}, are given in Table \ref{rulesQK}. For each logic {\bf Q.L} extending {\bf Q.K}, the calculus {\bf G3Q.L} is obtained by extending {\bf G3Q.K} with the non-logical rules of Table \ref{nonlogicalQK} that express proof-theoretically the geometric semantic properties  which define {\bf Q.L} (cf. Table \ref{prop}). Whenever a calculus contains rule $Eucl$, it contains also all its contracted instances $Eucl^c$ (see \cite[p. 100]{NP11}). Observe that $CBF$ ($BF$) is not derivable in calculi where rule $Incr$ ($Decr$) is not primitive nor admissible (given Proposition \ref{properties}.8, this can be checked semantically).

%\begin{center}\begin{tabular}{cc}
%\hline\hline\noalign{\smallskip}
%\infer[\infrule Ref_\W]{{}\GSD}{w\RR w,{}\GSD}\qquad&
%\infer[\infrule{ Ser,\, u \text{ fresh} }]{{}\GSD,}{w\RR u,{}\GSD}\\\noalign{\smallskip}
%\infer[\infrule Trans]{w\RR v,v\RR u,{} \GSD}{w\RR u,w\RR v,v\RR u,{}\GSD}\qquad&
%\infer[\infrule Eucl]{w\RR v,w\RR u,{}\GSD,}{v\RR u,w\RR v,w\RR u,{}\GSD}\\\noalign{\smallskip}
%\infer[\infrule{ NonEm,\, z\text{ fresh}}]{{}\GSD}{z\in w,{}\GSD}&
%\infer[\infrule Cons]{{}\GSD}{x\in w,{}\GSD}\\\noalign{\smallskip}
%\infer[\infrule Incr]{x\in w,w\RR v,{}\GSD}{x\in v,x\in w,w\RR v,{}\GSD}&
%\infer[\infrule Decr]{x\in v,w\RR v,{}\GSD,w:\Box A}{x\in w,x\in v,w\RR v,{}\GSD,u:A}\\\noalign{\smallskip}
%\infer[\infrule Eucl^c]{w\RR v,{}\GSD,}{v\RR v,w\RR v,{}\GSD}\\\noalign{\smallskip}\hline\hline
%\end{tabular}\end{center}

A \emph{{\bf G3Q.L}-derivation} of a sequent ${}\GSD$ is a tree of sequents, whose leaves are initial sequents, whose root is ${}\GSD$, and which grows according to the rules of {\bf G3Q.L}. As usual, we consider only derivations of \emph{pure sequents} -- i.e., sequents where no variable has both free and bound occurrences. The \emph{height} of a {\bf G3Q.L}-derivation  is the number of nodes of its longest branch. We say that ${}\GSD$ is {\bf G3Q.L}-derivable (with height $n$),  and we write $\mathbf{G3Q.L}\vdash^{(n)}{}\GSD$, if there is a {\bf G3Q.L}-derivation (of height at most $n$) of ${}\GSD$ or of an alphabetic variant of $\GSD$. A rule is said to be \emph{(height-preserving) admissible} in {\bf G3Q.L}, if, whenever its premisses are {\bf G3Q.L}-derivable (with height at most $n$), also its conclusion is {\bf G3Q.L}-derivable (with height at most $n$). In each rule depicted in Tables \ref{rulesQK} and \ref{nonlogicalQK}, $\g$ and $\d$ are called \emph{contexts}, the formulas occurring in the conclusion are called \emph{principal}, and the formulas occurring  in the premisses only are called \emph{active}.

\begin{table}\caption{Rules of {\bf G3Q.K}}\label{rulesQK}
\begin{center}\scalebox{1.00000}{\begin{tabular}{ccc}
\hline\hline\noalign{\medskip}
{\bf initial sequents:}& \qquad&${}w:p,\GSD,w:p$, with $p$ atomic\\\noalign{\medskip}
{\bf logical rules:}&\\\noalign{\medskip}
\infer[\infrule L\bot]{{}w:\bot,\GSD}{{}\GSD}\\\noalign{\medskip}
\infer[\infrule L\wedge]{{} w:A\wedge B,\GSD}{{}w:A,w:B,\GSD}\qquad&&
\infer[\infrule R\wedge]{{}\GSD,w:A\wedge B}{{}\GSD,w:A&{}\GSD,w:B}\\\noalign{\medskip}
\infer[\infrule L\lor]{{} w:A\lor B,\GSD}{{}w:A,\GSD&w:B,\GSD}\qquad&&
\infer[\infrule R\lor]{{}\GSD,w:A\lor B}{{}\GSD,w:A,w:B}\\\noalign{\medskip}
\infer[\infrule L\to]{{} w:A\to B,\GSD}{{}\GSD,w:A&w:B,\GSD}\qquad&&
\infer[\infrule R\to]{{}\GSD,w:A\to B}{{}w:A,\GSD,w:B}\\\noalign{\medskip}
\infer[\infrule L\forall]{y\in w,{}w:\forall xA,\GSD}{w:A[y/x],y\in w,w:\forall xA,\GSD}&&
\infer[\infrule{ R\forall, \text{ $z$ fresh}}]{{}\GSD,w:\forall xA}{z\in w,{}\GSD,w:A[z/x]}\\\noalign{\medskip}
\infer[\infrule{ L\exists, \text{ $z$ fresh}}]{w:\exists xA,\GSD}{z\in w,w:A[y/x],\GSD}&&
\infer[\infrule{ R\exists}]{y\in w,\GSD,w:\exists x A}{y\in w,{}\GSD,w:A[y/x]}\\\noalign{\medskip}
\infer[\infrule L\Box]{w\RR v,{}w:\Box A,\GSD}{v:A,w\RR v,w:\Box A,\GSD}&&
\infer[\infrule{ R\Box,\text{ $u$ fresh}}]{{}\GSD,w:\Box A}{w\RR u,{}\GSD,u:A}
\\\noalign{\medskip}
\infer[\infrule{ L\Diamond,\text{ $u$ fresh}}]{w:\Diamond A,\GSD}{w\RR u,{}u:A,\GSD}&&
\infer[\infrule R\Diamond]{w\RR v,\GSD,w:\Diamond A}{w\RR v,{}\GSD,w:\Diamond A, v:A}
\\\noalign{\medskip}\hline\hline
\end{tabular}}\end{center}
\end{table}
\begin{table}\caption{Non-logical rules}\label{nonlogicalQK}
\begin{tabular}{ccc}
\hline\hline\noalign{\medskip}
\infer[\infrule Ref_\W]{{}\GSD}{w\RR w,{}\GSD}\qquad&
\infer[\infrule Eucl]{w\RR v,w\RR u,{}\GSD,}{v\RR u,w\RR v,w\RR u,{}\GSD}&
\infer[\infrule Eucl^c]{w\RR v,{}\GSD,}{v\RR v,w\RR v,{}\GSD}
\\\noalign{\smallskip}
\infer[\infrule{ Ser,\, u \text{ fresh} }]{{}\GSD,}{w\RR u,{}\GSD}&
\infer[\infrule Trans]{w\RR v,v\RR u,{} \GSD}{w\RR u,w\RR v,v\RR u,{}\GSD}\\\noalign{\smallskip}
\end{tabular}

\begin{tabular}{ccc}
%\infer[\infrule{ NonEm,\, z\text{ fresh}}]{{}\GSD}{z\in w,{}\GSD}\\\noalign{\smallskip}
\infer[\infrule Incr]{x\in w,w\RR v,{}\GSD}{x\in v,x\in w,w\RR v,{}\GSD}&
\infer[\infrule Decr]{x\in v,w\RR v,{}\GSD}{x\in w,x\in v,w\RR v,{}\GSD}&
\infer[\infrule Cons]{{}\GSD}{x\in w,{}\GSD}%\\\noalign{\smallskip}
\\\noalign{\medskip}\hline\hline
\end{tabular}
\end{table}

The following proposition presents the main meta-theoretical properties of {\bf G3Q.L}. The  proofs can be found in \cite[Chap.~12.1]{NP11}.
\begin{proposition}[Properties of {\bf G3Q.L}]\label{properties}\
\begin{enumerate}
\item Sequents of shape %${}w:\bot,\GSD$ or ${}\GSD,w:\top$ or
 ${}w:A,\GSD,w:A$ (with $A$ non-atomic) are {\bf G3Q.L}-derivable.
\item The rule of $\alpha$-conversion is height-preserving admissible: if {\bf G3Q.L} $\vdash^n {}\GSD$, then {\bf G3Q.L} $\vdash^n{}\g'\To\d'$, where $\g'$ ($\d'$) is obtained from $\g$ ($\d$) by renaming bound variables.
\item The following rules of substitution are height-preserving admissible in {\bf G3Q.L}:
$$
\infer[\infrule{ [y/x]}]{\g[y/x]\To\d[y/x]}{{}\GSD}\qquad
\infer[\infrule{ [w/v]}]{\g[w/v]\To\d[w/v]}{{}\GSD}
$$
where $y$ is free for $x$ in each formula occurring in $\g,\d$ for rule $[y/x]$.
\item The   following rules of weakening are height-preserving admissible in {\bf G3Q.L}:
$$
\infer[\infrule LW]{\g',\GSD}{\GSD}\qquad
\infer[\infrule RW]{\GSD,\d'}{\GSD}
$$
\item Each rule of {\bf G3Q.L} is height-preserving invertible.
\item  The following rules of contraction are height-preserving admissible in {\bf G3Q.L}:
$$
\infer[\infrule LC]{{}\g',\GSD}{{}\g',\g',\GSD}\qquad
\infer[\infrule RC]{{}\GSD,\d'}{{}\GSD,\d',\d'}
$$
\item The following rule of Cut is admissible in {\bf G3Q.L}:
$$
\infer[\infrule Cut]{\g',\GSD,\d'}{{}\GSD,w:A&w:A,\g'\To\d'}
$$
\item  {\bf G3Q.L} is sound and complete with respect to  {\bf Q.L}.
\end{enumerate}
\end{proposition}

\section{Quantified modal logics with definite descriptions}\label{sec:definite}
\subsection{Preliminary discussion}\label{secintrdd}
Before introducing the formal machinery used by Fitting and Mendelshon \cite{FM98} to deal with definite descriptions, we take a minute to outline some of the main semantic  approaches to definite descriptions, see \cite{I18,I19} for more details. Given a formula with one free variable $A(x)$,  let us consider the description:
\begin{equation}\label{description}
\textnormal{the $x$ such that }A(x)
\end{equation}

First of all, we can simply deny that  (\ref{description}) has to be represented by a genuine term of the formal language. In this case, following Russell \cite{R05}, we can  explain it away   by means of of  quantification and identity as follows:
$$%\begin{equation}\label{russell}
\exists x(A(x)\wedge\forall y(A(y)\to y=x))
$$%\end{equation}
This is a very simple solution in that we simply get rid of the problem of giving a semantics for improper definite descriptions. 

If, instead, following Frege \cite{F92} we maintain that  (\ref{description}) has to be represented by a genuine term of the formal language, we extend the language with terms of  shape:
$$%\begin{equation}\label{iotaterm}
\riota xA(x)
$$%\end{equation}
and we have to give a satisfactory semantics  for descriptions. If  a description is proper, then it denotes the one and only existing object  that satisfies it (in a given world). The  problem is what to do when a description is improper  because  either it is true of no object --e.g., $\riota x(x\neq x)$ -- or it is true of more than one object --e.g., $\riota x(x=x)$.
If we don't want to extend the language then we cannot allow for non-denoting terms. Thus, we have to accept that improper descriptions denote some object.  In a constant domain setting -- i.e., when we use the quantification theory of classical logic -- we can follow     Montague and Kalish \cite{KM57} and  assume that all improper descriptions denote a chosen object. If, instead, we are using  a varying domain semantics -- i.e.,  the quantification theory of positive free logic -- then we can follow Garson \cite{G13} and assume that each improper description denotes a non-existing object (without thereby assuming  it is the same one for all improper descriptions). 
Nevertheless, these two solutions are not satisfactory in that it is more natural to maintain that an improper description simply fails to denote. Moreover, like  Russell's approach,  they do not disentangle designation from existence. 
%To witness, Kalish and Montague's definite descriptions are axiomatixed by:
%\begin{description}
%\item[$KM_1.$]\qquad $\forall y\forall x(A(x)\tto x=y)\to y=\riota xA(x)$
%\item[$KM_2.$]\qquad $\neg\exists y\forall x(A(x)\to x=y)\to y=\rotatebox[origin=c]{180}{t}$
%\end{description}
%where $KM1$ axiomatizes proper description and $KM2$ improper ones. Garson's proper descriptions, instead, are axiomatized by:
%\begin{description}
%\item[$G_1.$] \qquad$(\mathcal{E} y\wedge y=\riota xA(x))\to \forall x(A(x)\tto x=y)$
%\item[$G_2.$] \qquad$(\mathcal{E} y\wedge\forall x(A(x)\tto x=y) )\to y=\riota xA(x)$
%\end{description}
%where $\mathcal{E}y$ is shorthand for $\exists z(y=z)$, and where we have no axiom for improper descriptions.

The addition of $\lambda$-abstraction to the language allows Fitting and Mendelsohn \cite{FM98} to avoid these shortcomings: proper descriptions denote the one and only object that satisfies them (be it an existing object or not) and improper descriptions simply fails to denote. Failure of denotation will not be a problem because definite descriptions do not occur as \emph{relata} of atomic formulas, but only as terms applied via $\lambda$. Hence we can simply impose that if $t$ is an improper description then $\l x A.t$ is false: this implies  that $\neg (\l x A.t)$ is true -- i.e., that  the sentence $At$ is false -- without thereby implying  that $\l x \neg A.t$ is true -- i.e., that  the negation of $A$ is true of the object denoted by $t$. 
\subsection{Syntax and semantics}\label{sec:descsem}
%Let us extend the signature  $\mathcal{S}$ introduced in  Section \ref{sect:syntax} with an at most denumerable set $CON^\mathcal{S}$ of individual constants $a,b,c\dots$ (functions of higher arity are omitted for simplicity). 
Let us consider the same signature $\mathcal{S}$ of Section \ref{sect:syntax} (functions of any arity are omitted for simplicity).
The sets of terms and  formulas of the language $\mathscr{L}^\l$ are defined simultaneously as follows:
$$
t::= x\;|\;\riota xA\phantom{\;|\;\;\\;\;; \;\;\;\;x_1=x_2\;|\; \bot\;|\;A\wedge A\;|\;A\lor A\;|\;A\to A\;|\;\forall xA\;|\;\exists xA\;|\;\Box A\;|\;\Diamond A\;|\; \l x_1A.t}
$$
$$
A::=P^nx_1,\dots,x_n\;|\; x_1=x_2\;|\; \bot\;|\;A\wedge A\;|\;A\lor A\;|\;A\to A\;|\;\forall xA\;|\;\exists xA\;|\;\Box A\;|\;\Diamond A\;|\; \l xA.t
$$
where $P_n\in REL^\mathcal{S}$ %, $a\in CON^\mathcal{S}$, 
and $x,x_1,\dots,x_n\in VAR$. Observe that definite descriptions can occur in a formula only as terms applied by the operator $\l$.   We continue to use the conventions and notions introduced in Section \ref{sect:syntax} with the following additions: 
\begin{itemize}
\item
in $\l xA.t$ all occurrences of $x$ (save for the displayed $t$ in case $t\equiv x$) are\mbox{ bound by $\l x$;}
%\item$a,b,c$ range over individual constants;
\item $t,r,s$ range over terms.
\end{itemize}

Frames, models and assignments are defined as in Section \ref{sect:syntax}.
% with the following additional clause for individual constants:
%\begin{itemize}
%\item for each individual constants $c$ and for some, possibly not all, $w\in \W$, $\V(c,w)\in D$.
%\end{itemize}
%Observe that this clause entails that constants are non rigid terms -- i.e., they might denote different objects in different (but accessible) worlds -- and they might not denote an object in some world. 
Because of definite descriptions, we have to define the notions of denotation and  satisfaction together.
\begin{definition}[Denotation and satisfaction] Given a model $\MM$, an assignment $\sigma$ over it, and a world $w$ of that model, we define the notions of \emph{denotation} of a term $t$ and  \emph{satisfaction} of an $\mathscr{L}$-formula $A$ as follows:
\begin{itemize}
\item Denotation of a term $t$:

\begin{tabular}{lll}\noalign{\smallskip}
$\V_w^\s(x)=\s(x)$\\\noalign{\smallskip}
%$\V_w^\s(a)=o$&iff&$o= \V(a,w)$\\\noalign{\smallskip}
$\V_w^\s(\riota xA)=o$&iff&$o$ is the one and only member of $D_\W$ such that $\s^{x\tr o}\models_w^\M A$\\\noalign{\smallskip}
\end{tabular}
\item Satisfaction is defined by extending Definition \ref{sat} with the following clauses:

\begin{tabular}{lll}\noalign{\smallskip}
$\s\models_w^\M x=y$&iff &$\s(x)=\s(y)$\\\noalign{\smallskip}
$\s\models_w^\M \l x A.t$&iff&$\V_w^\s(t)$ is defined and $\s^{x\tr\V_w^\s(t)}\models_w^\M A$\\\noalign{\smallskip}
\end{tabular}
\end{itemize}
\end{definition}
Truth and validity are defined as in Section \ref{sect:syntax}. All correspondence results of Section \ref{noid} carries over to QMLs with definite descriptions.

Finally, we show the generality of this approach to definite descriptions by showing how it can simulate the other (non-eliminative)  approaches  considered in Section \ref{sec:definite}. 

\begin{proposition}\label{simulate}\
\begin{enumerate}
\item A Montague and Kalish-style \cite{KM57} description can be  defined as:
$$
\rotatebox[origin=c]{180}{$\iota$}_mxA\;\equiv \riota x (\exists !y(A[y/x])\veebar Ux)
$$
where $\exists !$ is the unique existential quantifier, $\veebar$ is exclusive disjunction and $U$ is a constant  monadic predicate axiomatized by $\forall y(U(y)\to\exists z(y=z))$ and \mbox{$\forall z(U(y)\wedge U(z)\to y=z)$} (our language does not contain contants).
\item A Garson-style \cite{G13} description can be defined as:
$$
\rotatebox[origin=c]{180}{$\iota$}_gxA\;\equiv \riota x (\exists !y(A[y/x])\veebar U_{\riota xA}(x) )
$$
where $U_{\riota x A}$ is a constant  predicate, that is  parametric on  $\riota xA$ (modulo alphabetic variants), such that $\forall z(U_{\riota xA}(y)\wedge z=y\to\bot)$ and  $U_{\riota xA }(z)\wedge U_{\riota xA} (y)\to z=y$.
\end{enumerate}
\end{proposition}% Finally, correspondence results for validity over rigid and total models are presented in Table \ref{propl}.
%
% \begin{table}\caption{Additional modal axioms and corresponding properties}\label{propl}
% 
%\begin{center}\scalebox{0.92000}{\begin{tabular}{ll}
%\hline\hline\noalign{\smallskip}
%$RG:= \l x.\Box A.t\to \Box(\l x.A.t)$&$ \forall w,v\in \W, \forall c\in CON^\mathcal{S}(w\R v\& \V(w,c)\in D\to \V(v,c)= \V(w,c)$) \\\noalign{\smallskip}
%$TT:=\l x.x=x.t$&$\forall w\in\W\forall c\in CON^\mathcal{S}( \V(w,c)\in D)$\\\noalign{\smallskip}\hline\hline
%\end{tabular}}\end{center}
%\end{table} 
\section{Labelled calculi}\label{sec:desclab}
In order to introduce labelled sequent calculi for QMLs with definite descriptions, we  extend the language of labelled calculi with \emph{denotation formulas} of shape $D(t,x,w)$, which will be used to express that  the variable  $x$ denotes the object  denoted in $w$  by the term  $t$. From now on, a sequent $\GSD$  is an expression where $\g$ is a multiset of labelled  $\mathcal{L}^\l$-formulas and of formulas of shape $D(t,x,w)$, $x\in w$ or $w\RR v$; and  $\d$ is a multiset  of labelled $\mathcal{L}^\l$-formulas and of denotation formulas only. The following non-standard definition of weight will be essential in Sections \ref{sec:structural} and \ref{sec:characterization}.

\begin{definition}[Weight of terms and formulas] \
\begin{itemize}
\item The weight of a term $t$ is $0$ if $t$ is a variable  and, if $t\equiv \riota xA$, it is equal to the weight of $w:\riota xA$;
\item The weight of a labelled $\mathcal{L}^\l$-formula $w:A$ is defined as the number of operators that differs from $\bot$ (and  $=$) occurring in $A$ plus the weight of each occurrence of a term in $A$;
\item The weight of a formula $D(t,x,w)$ is %$\den{t}{x}{w}$ 
equal to the weight of the term $t$;
\item the weight of formulas of shape $x\in w$ and  $w\RR v$ is $0$.
\end{itemize}
\end{definition}

The rules of the calculus {\bf G3Q$\l$.L} are the rules of {\bf G3Q.L}, see Tables \ref{rulesQK} and \ref{nonlogicalQK}, plus the initial sequents and  rules given in Table \ref{rulesQKld}. Observe that the rules for identity contain the labelled version of the non-logical rules  first introduced in \cite{NP98}. When $w:y=x$ holds, by $Repl$ we can replace $x$ with $y$ in any atomic formula that, so to say, talks about $w$. Rule $RigVar$ implies that if $x$ and $y$ denote the same object in some world, they do so in each world. Thus, variables  behave as rigid designators and   labels could be omitted from identities.  We choose to keep them in order to have a more uniform notation.

 \begin{table}\caption{Additional rules for $\mathbf{G3\l.L}$}\label{rulesQKld}

\begin{center}\scalebox{1.000000}{\begin{tabular}{ll}
\hline\hline\noalign{\smallskip}
{\bf Initial sequents:} & $D(y,x,w),\GSD,D(y,x,w)$
\\\noalign{\smallskip}\hline\noalign{\smallskip}
{\bf Rules for identity:}\phantom{aaaaaaaaaaaaaaaaa}&
\\\noalign{\smallskip}
\infer[\infrule Ref_=]
{\GSD}{{}w:x=x,\GSD}&
\infer[\infrule RigVar]{{} w:y=z,\GSD}{{}v:y=z,w:y=z,\GSD}
\\\noalign{\smallskip}
\infer[\infrule Repl]{E[y/x],w:y=z,\GSD,}{E[z/x],E[y/x],w:y=z,\GSD}
&\deduce[\text{$E$ is either $D(y,x,w)$  or $x_i\in w$ or $w:p$}]{\phantom{A} }{\phantom{ A} }
\\\noalign{\smallskip}\hline\noalign{\smallskip}
{\bf Rules for $\l$:}\phantom{aaaaaaaaaaaaaaaaaaaaaaa}&
\\\noalign{\smallskip}

%\multicolumn{2}{c}{\infer[\infrule{ L\l,\, z\text{ fresh}}]{{}w:\l xB.\riota yA,\GSD}{D(\riota yA,z,w),w:B[z/x],\GSD }}\\\noalign{\medskip}
%\multicolumn{2}{c}{\infer[\infrule R\l]{\GSD,w:\l xB.\riota yA}{\GSD,w:\l xB.\riota yA, D(\riota yA,z,w)\quad&\GSD,w:\l xB.\riota yA,w:B[z/x]} }
\multicolumn{2}{c}{\infer[\infrule{ L\l,\, z\text{ fresh}}]{{}w:\l xB.t,\GSD}{D(t,z,w),w:B[z/x],\GSD }}\\\noalign{\medskip}
\multicolumn{2}{c}{\infer[\infrule R\l]{\GSD,w:\l xB.t}{\GSD,w:\l xB.t, D(t,y,w)\quad&\GSD,w:\l xB.t,w:B[y/x]} }

\\\noalign{\smallskip}\hline\noalign{\smallskip}

{\bf Rules for $D(\dots)$:}\phantom{aaaaaaaaaaaaaaaaa}\\\noalign{\smallskip}
\multicolumn{2}{c}{\infer[\infrule LD_1]{D(\riota x_1A,x_2,w),\GSD}{w:A[x_2/x_1],D(\riota x_1A,x_2,w),\GSD} }\\\noalign{\medskip}
\multicolumn{2}{c}{\infer[\infrule LD_2]{D(\riota x_1A,x_2,w)\GSD}{D(\riota x_1A,x_2,w),\GSD,w:A[y/x_1]\quad&w:x_2=y,D(\riota x_1A,x_2,w),\GSD} }
\\\noalign{\medskip}
\multicolumn{2}{c}{\infer[\infrule{RD,\, z\text{ fresh}}]{\GSD,D(\riota x_1A,x_2,w)}{\GSD,w: A[x_2/x_1]\qquad&w:A[z/x_1],\GSD,w: x_2=z}
}\\\noalign{\medskip}
\infer[\infrule DenVar]{{}\GSD}{D(x,x,w),{}\GSD}&
\infer[\infrule DenId]{D(y,x,w),{}\GSD}{w:y=x,D(y,x,w),\GSD}
%\\\noalign{\smallskip}\hline\noalign{\smallskip}
%{\bf non-logical rules:} ($c\in CON^\mathcal{S}$)\phantom{aaaaaa }&\\\noalign{\smallskip}
%\infer[\infrule{ Tot,\,z\text{ fresh}}]{{}\GSD}{D(c,z,w),{}\GSD}&
%\infer[\infrule Rig]{D(c,x,w),w\RR v,{}\GSD}{D(c,x,v),D(c,x,w),w\RR v,{}\GSD}
\\\noalign{\smallskip}\hline\hline 
\end{tabular}}\end{center}
\end{table}

 The satisfaction clause for $\l x.A.t$ in a world $w$ is similar to  that for  $\exists xA$, the only difference being that $A$ has to be satisfied not by some arbitrary object of $D_w$, but by the one and only object of $D_\W$ that is denoted by $t$ in that world of that model. Therefore the rules for $\l$ are like the ones for $\exists$ in intuitionistic logic  with existence predicate, see  \cite{BI06,TS96}, save that they are restricted by  formulas of shape $D(t,x,w)$ instead of by atoms of shape $E t$. 
 
 Next, we briefly expalin the rules for $D(t,x,w)$. The  universal rule $DenId$ ensures  that if $y$ in $w$ picks the object denoted by $x$, then $x$ and $y$ denote the same object; and  the universal rule $DenVar$ ensures that variables denote at every world. The rules $LD_i$ and $RD$ are obtained as meaning-explanation of the denotation clause for definite descriptions.  This is done by first rewriting the denotation clause for $\riota x_1 A$ as:
 
\begin{tabular}{lll}\noalign{\smallskip}
$\V_w^\s(\riota xA)=o_1$&iff&$\s^{x\tr o_1}\models_w^\M A$ and $\forall o_2\in D\,(\s^{x\tr o_2}\models_w^\M A\to o_2=o_1)$ \\\noalign{\smallskip}
\end{tabular}

\noindent Then, from the left-to-right (right-to-left) direction of this semantic clause we easily obtain the rules $LD_i$ ($RD$). 

As shown in \cite{MO19}, for intuitionistic logic with existence predicate it is possible to obtain a simpler calculus by replacing the two premisses  rule \infer[\infrule R\exists^*]{\g\To \exists xA}{\g\To  E t&\g\To A[t/x]} with the (equivalent) one premiss rule \infer[\infrule R\exists]{E t,\g\To\exists xA}{E t,\g\To A[t/x]}.  The same phenomenon holds form QMLs with varying domains where we can use either two-premisses versions of rules $R\exists$ and $L\forall$ \cite{I18} or the simpler one-premiss versions thereof \cite{NP11}. In Table \ref{rulesQK} we have used the one-premiss rules and the same has been done for the rule $R\l$ in \cite{OC18}. Here, instead, we are forced to adopt the two-premisses version of the rule $R\l$ (and of $LD_2$) because the presence of definite description impairs the admissiblity of cut  with the one-premiss version of this rule. The problem, roughly, is that $D(\riota x_1A, x_2,w)$ (or $w:A[y/x_1]$ for rule $LD_2$) is not atomic and, therefore, cannot be a principal formula of the rule.

%\begin{dubbio}
%Mi serve la seguente regola (che impone che la denotazione sia funzionale)?
%$$
%\infer{D(t,x,w),D(t,y,w),\GSD}{w:x=y,D(t,x,w),D(t,y,w),\GSD}
%$$
%e, se mi serve, posso restringerla ai casi in cui $t$ non \`e una descrizione definita?
%
%  Par le variabili  segue da \emph{DenId}, ma per gli altri termini non mi pare derivabile (sebbene forse non si dia mai un sequente con le due formule principali della conclusione)
%\end{dubbio}
%
%\begin{dubbio}
%Posso fare a meno della seguente istanza non atomica  di \emph{Repl}:
%
%$$
%\infer{D(\riota xA,z,w), w:y_1=y_2,\GSD}{D(\riota xA,z,w)[y_2/y_1],D(\riota xA,z,w), w:y_1=y_2,\GSD}
%$$
%\end{dubbio}
%
\section{Structural properties}\label{sec:structural}
\begin{lemma}[Initial sequents]\label{ax} Let $A$ be an  arbitrary $\mathcal{L}^\l$-formula. Sequents of the following shapes  are {\bf G3Q$\l$.L}-derivable: 
\begin{enumerate}
\item\qquad${}w:A,\GSD,w:A$
\item \qquad$D(\riota yA,z,w),\GSD,D(\riota yA,z,w)$ 
\end{enumerate}
\end{lemma}
\begin{proof} The two cases are proved by simultaneous  induction on the weight of the principal formula. For the inductive steps it is enough to apply, root first, the rules for the outermost operator ($D$ included) of the principal formula and then  the inductive hypothesis (IH). To illustrate, for $D(\riota yA,z,w),\GSD,D(\riota yA,z,w)$ we have:

\noindent$$\scalebox{0.80000}{
\infer[\infrule RD]{D(\riota xA,z,w),\GSD,D(\riota xA,z,w)}{
\infer[\infrule LD_1]{D(\riota xA,z,w),\GSD,w:A[z/x]}{\infer[\infrule IH]{w:A[z/x],\dots\To w:A[z/x]}{}}&
\infer[\infrule LD_2]{w:A[y/x],D(\riota xA,z,w),\GSD,w:z=y}{
%\phantom{\g}&
\infer[\infrule IH]{w:A[y/x],\dots\To\dots,w:A[y/x]}{\phantom{\g}}&
w:z=y,\dots\To w:z=y}}
}$$

\end{proof}

\begin{lemma}[$\alpha$-conversion]\label{alpha}
{\bf G3Q$\l$.L} $\vdash^n {}\GSD$ entails {\bf G3Q$\l$.L} $\vdash^n {}\g'\To\d'$, where $\g'$ ($\d'$) is  obtained from $\g$ ($\d$) by renaming some bound variable (without capturing variables).
\end{lemma}
\begin{proof}
The proof is by induction on the height of the {\bf G3Q$\l$.L}-derivation $\D$ of ${}\GSD$. To illustrate, suppose we know that {\bf G3Q$\l$.L} $\vdash^n {}w:\l x.A.t,\GSD$, and we want to show that {\bf G3Q$\l$.L} $\vdash^n {}w:\l y.A[y/x].t,\GSD$ (with $y$ fresh). If $w:\l x.A.t$ is not principal in the last step of $\D$, the proof  is straightforward. Else, we transform
$$
\infer[\infrule L\l]{{}w:\l x.A.t,\GSD}{D(t,z,w),{}w:A[z/x],\GSD}\qquad\text{into}\qquad
\infer[\infrule L\l]{{}w:\l y.A[y/x].t,\GSD}{\infer[\star]{D(t,z,w),{}w:(A[y/x])[z/y],\GSD}{D(t,z,w),{}w:A[z/x],\GSD}}
$$
where the step $\star$ is height-preserving admissible since, having assumed that $y$ is fresh,  $w:(A[y/x])[z/y]$ is just a  cumbersome notation for $w:A[z/x]$. 
\end{proof}
\begin{lemma}[Substitutions]\label{subs}The following rules  of substitution are height-preserving admissible in {\bf G3Q$\l$.L}:
$$
\infer[\infrule{ [y/x]}]{{}\g[y/x]\To\d[y/x]}{{}\GSD}\qquad
\infer[\infrule{ [w/v]}]{{}\g[w/v]\To\d[w/v]}{{}\GSD}
$$
where $y$ is free for $x$ in each formula occurring in $\g,\d$ for rule $[y/x]$.
\end{lemma}
\begin{proof} Both proofs are by induction on the height of the derivation $\D$ of the premiss ${}\GSD$. The  base cases and  the inductive steps where the last rule is not a rule from Table \ref{rulesQKld} are proved in \cite[Lemma 12.4]{NP11}. 

We consider explicitly only he case of rule $[y/x]$ where the last step is by $L\l$ and the substitution $[y/x]$ clashes with its variable condition. E.g.,  the last step of $\D$ is
$$
\infer[\infrule L\l]{{}w:\l z.A.t,\g'\To\d}{D(t,y,w),w:A[y/z],\g'\To\d}
$$
with $x$ occurring free in $w:A[y/z],\g',\d$ and/or $t\equiv x$. We apply IH twice to the premiss of the last step of $\D$, the first time to replace $y$ with $y'$, for some fresh variable $y'$, and the second time to replace $x$ with $y$. We finish by applying rule $L\l$. Thus, assuming  $z\not\equiv x$,  we have transformed $\D$ into $\D[y/x]$:

$$
\infer[\infrule L\l]{w:\l z.(A[y/x]).(t[y/x]),\g'[y/x]\To\d[y/x]}{\infer[\infrule \star]{D(t[y/x],y',w),w:(A[x/y])[y'/z],\g'[y/x]\To\d[y/x]}{\infer[\infrule IH]{D(t[y/x],y',w),w:(A[y'/z])[y/x],\g'[y/x]\To\d[y/x]}{\infer[\infrule IH]{D(t,y',w),{}w:A[y'/z],\g'\To\d}{D(t,y,w),w:A[y/z],\g'\To\d}}}}
$$

\noindent which has the same height as $\D$ because the steps by IH are height-preserving admissible and the step by $\star$ is an height-preserving admissible  rewriting  that is feasible because $z\not\equiv x$ and $y'\not\in \{ y,x\}$.
\end{proof}

\begin{theorem}[Weakening]\label{weak}The   following rules  are height-preserving admissible in {\bf G3Q$\l$.L}:
$$
\infer[\infrule LW]{\g',\GSD}{\GSD}\qquad
\infer[\infrule RW]{\GSD,\d'}{\GSD}
$$
\end{theorem}
\begin{proof} The proofs are by induction on the height of the derivation $\D$ of the premiss ${}\GSD$. The  base cases and  the inductive cases where the last step of $\D$ is not by a rule from Table \ref{rulesQKld} are proved in \cite[Thm.~12.5]{NP11}. The proofs of the inductive cases when the last step of $\D$ is by $L\l$ or by $R D$ ($R\l$ or $LD_i$) are analogous to the ones in \cite[Thm.~12.5]{NP11} with last step of $\D$ by rule $L\exists$  ($R\exists$, respectively). The remaining cases are similar to the other ones  by  geometric rules and can be omitted.
\end{proof}
\begin{lemma}
The rules for $\l$ and for $D(\dots)$ are invertible.
\end{lemma}

\begin{proof}
Rules $R\l$ ,$L D_i$, \emph{DenVar}, and  \emph{DenId} are `Kleene'-invertible thanks to the repetition of the principal formula in the premiss(es).

Let's consider  $R D$. Suppose we have a $\mathbf{G3Q\l.L}$-derivation $\mathcal{D}$ of height $n$ of the sequent $\GSD,D(\riota xA,z,w)$. If $n=0$ or if $D(\riota xA,z,w)$ is principal in the last step of $\mathcal{D}$, the lemma holds trivially. Else we reason by cases on the last step $R$ of $\mathcal{D}$ and we transform the derivation(s) of its premiss(es) $\g_i\To\d_i,D(\riota xA,z,w)$ ($i\in\{1,2\}$) as follows: first, if $R$ has a variable condition, we apply an hp-admissible instance of substitution to ensure it differs from $y$; then we apply the inductive hypothesis to obtain either $\g'_i\To\d'_i,w:A[z/x]$ or $w:A[y/x],\g'_i\To\d'_i,w:z=y$; finally we apply an instance of $R$. 

Rule $L\l$ can be treated analogously.
\end{proof}
\begin{corollary}[Invertibility]\label{inv} Each rule of {\bf G3Q$\l$.L} is height-preserving invertible.
\end{corollary}
%\begin{proof} We prove only the case of $L\l$ ($R\l$ is `Kleene'-invertible thanks to the repetition of the principal formulas in the premiss). The proof is by induction on the height of the derivation $\D$ of ${}w:\l x.A.t,\GSD$. If the height of $\D$ is 0 or if $w:\l x.A.t$ is principal in the last step of $\D$, the lemma holds trivially. In the other inductive cases we obtain a derivation of  $D(t,y,w),{}w;A[y/x],\GSD$ (where $y$ is any variable) by first applying to the premiss(es) of the last step of $\D$ an height-preserving admissible instance of substitution to avoid problems with variable conditions on the last step of $\D$, if this is needed. Then, we apply IH to the sequent(s) just obtained and we finish by applying the last rule applied in  $\D$. 
%\end{proof}
\begin{theorem}[Contraction]\label{cont} The following rules  are height-preserving admissible in {\bf G3Q$\l$.L}:
$$
\infer[\infrule LC]{\g',\GSD}{\g',\g',\GSD}\qquad
\infer[\infrule RC]{\GSD,\d'}{\GSD,\d',\d'}
$$
\end{theorem}
\begin{proof} The proof is handled by a simultaneous induction on the height of the derivations of the  premisses of  $LC$ and  $RC$. Without loss of generality, we assume the multiset we are contracting is made of only one formula $E$.

The base cases obviously hold, and the inductive cases depend on whether zero,  one, or two instances of  $E$ are principal in the last step $R$ of the derivation $\D$ of the premiss. If zero instances are principal in $R$, we apply IH to the premiss(es) of $R$ and then an instance of rule $R$, and we are done. 

If one instance is principal and $R$ is by a propositional rule or by one of $R\forall,L\exists, R\Box,L\Diamond, L\l$ and $R D$, we proceed by first applying invertibility to that rule, then we apply IH as many times as needed, and we conclude by applying an instance of that rule. If, instead, one instance is principal and $R$ is by a rule with repetition of the principal formula(s) in the premiss, we use the hp-admissibility of the rules of weakening, then IH and $R$.

If two instances are principal,  $R$ is a geometric rule and, if needed, we make use of the fact that $R$ satisfies the closure condition (see \cite[p. 100]{NP11}). To illustrate, the case of $Euclid$ is taken care by the presence of its contracted instances $Euclid^c$. For $Trans$,  we have three  occurences of $w\RR w$  in the premiss of this rule instances: two principal and one active. We apply IH twice and we are done. For $Repl$, the active formula of the last rule instance must be of shape $w:x=x$, and, after having applied IH, we can get rid of it by applying $Ref_=$.
\end{proof}

\begin{theorem}[Cut]\label{cut}The following rules of Cut are admissible in {\bf G3Q.L}:
$$
\infer[\infrule Cut]{\g',\GSD,\d'}{\GSD,w:A&w:A,\g'\To\d'}
$$
$$
\infer[\infrule Cut']{\g',\GSD,\d'}{\GSD,D(t,x,w)&D(t,x,w),\g'\To\d'}
$$
\end{theorem}
\begin{proof} We prove the two cases simultaneously. The proof, which extends that of \cite[Thm.~12.9]{NP11}, considers an uppermost instance of  either $Cut$ or $Cut'$  which is handled by a principal induction on the weight of the cut-formula  with a sub-induction on the sum of the heights of the derivations $\D_1$ and $\D_2$ of the two premisses of cut (\emph{cut-height}, for shortness). 
The proof can be organized in four exhaustive cases: in {\bf case 1} one of the two premisses is an initial sequent. In {\bf case 2} the cut formula is not principal in the left premiss only and in {\bf case 3} it is not principal in the right premiss. Finally, in {\bf case 4}, the cut formula is principal in both premisses.
%\begin{enumerate}
%\item one of the two premisses is an initial sequent; 
%\item the cut-formula is not principal in the left premiss only;
%\item the cut-formula is not principal in the right premiss;
%\item the cut-formula is principal in both premisses.
%\end{enumerate}

In {\bf case 1}, the conclusion of $Cut$ ($Cut'$) is an initial sequent and, therefore, we can dispense with that instance of $Cut$ ($Cut'$).

In {\bf case 2}, we transform the derivation as follows. First, if the last rule applied in $\D_1$ has a variable condition, we apply an height-preserving admissible substitution to rename its \emph{eigenvariable} with a fresh one. Then,  we apply one  instance of $Cut$ ($Cut'$) on  each premiss of $\D_1$ with the conclusion of $\D_2$. These instances of $Cut$  ($Cut'$) are admissible by IH because they have a lesser cut-height. We finish  by applying an instance of the last rule applied in $\D_1$.

{\bf Case 3} is similar to case 2.  To illustrate, suppose the last step of $\D_2$ is by $L\l$ (with $y$ \emph{eigenvariable}), we transform
$$
\infer[\infrule Cut]{v:\l x.B.t,\g',\GSD,\d'}{{}\GSD,w: A&\infer[\infrule L\l]{w:A,v: \l x.B.t,\g'\To\d'}{D(t,y,v),w:A,v:B[y/x],\g'\To\d'}}
$$
into
$$
\infer[\infrule L\l]{v:\l x.B.t,\g',\GSD,\d'}{\infer[\infrule Cut]{D(t,z,v),v:B[z/x],\g,\g'\To\d,\d'}{{}\GSD,w: A&\infer[\infrule{[z/y]}]{D(t,z,v),w:A,v:B[z/x],\g'\To\d'}{D(t,y,v),w:A,v:B[y/x],\g'\To\d'}}}
$$

In {\bf case 4}, we have subcases according to the principal operator of the cut-formula. We consider only the case where the cut-formula  is of shape $w:\l yB.t$  or  $D(\riota xA,z,w)$ (see \cite[Thm.~12.9]{NP11} for the other cases). 
\begin{itemize}
\item
Cut formula is $w: \l yB.t$. We transform 
\end{itemize}
$$\scalebox{0.90000}{
\infer[\infrule Cut]{\g',\GSD,\d'}{
\infer[\infrule R\l]{\g'\To\d', w:\l yB.t}{\deduce{\g'\To\d',w:\l yB.t,D(t,x,w)}{\vdots\;\mathcal{D}_{11}}&
\deduce{\g'\To\d',w:\l yB.t,w:B[x/y]}{\vdots\;\mathcal{D}_{12}}}&
\infer[\infrule L\l]{w:\l yB.t,\GSD}{\deduce{D(t,z,w),w:B[z/y],\GSD}{\vdots\;\mathcal{D}_{21}}}}
}$$
into:
$$\scalebox{0.96000}{
\infer=[\infrule LC+RC]{\g',\GSD,\d'}{\infer[\infrule Cut_2]{\g',\g',\g,\g,\GSD,\d,\d,\d',\d'}{\infer[\infrule Cut_1]{\g',\GSD,\d', w:B[x/y]}{
\deduce{\g'\To\d',w:B[x/y],w:\l yB.t}{\vdots\;\mathcal{D}_{12}}&
\deduce{w:\l yB.t,\GSD}{\vdots\mathcal{D}_2}
}&
\deduce{w:B[x/y],\g',\g,\GSD,\d,\d'}{\vdots\;\mathcal{D}_{3}}
}
}
}$$
where $\D_3$ is the following derivation:
$$\scalebox{0.96000}{
\infer[\infrule Cut'_4]{w:B[x/y],\g',\g,\GSD,\d,\d'}{\infer[\infrule Cut_3]{\g',\GSD,\d', D(t,x,w)}{
\deduce{\g'\To\d',D(t,x,w),w:\l yB.t}{\vdots\;\mathcal{D}_{11}}&
\deduce{w:\l yB.t,\GSD}{\vdots\mathcal{D}_2}
}&
\infer[\infrule{[x/z]}]{D(t,x,w),w:B[x/y],\GSD}{\deduce{D(t,z,w),w:B[z/y],\GSD}{\vdots\;\mathcal{D}_{21}}}
}
}$$
 \emph{Cut}$_1$ and \emph{Cut}$_3$ are  admissible because they have  has lesser cut-height; \emph{Cut}$_2$ and \emph{Cut}$'_4$ because their  cut-formula has lower weight.

\begin{itemize}
\item
Cut formula is $D(\riota xA,z,w)$.\vspace{0.2cm}
\end{itemize}

- If right premise is by $LD_1$, we have:

$$\scalebox{0.95000}{
\infer[\infrule Cut']{\g',\GSD,\d'}{
\infer[\infrule RD]{\g'\To\d',D(\riota xA,z,w)}{
\deduce{\g'\To\d', w:A[z/x]}{\vdots\;\mathcal{D}_{11}}&
\deduce{w:A[y/x],\g'\To\d' ,w:z=y}{\vdots\;\mathcal{D}_{12}}
}&
\infer[\infrule LD_1]{D(\riota xA,z,w),\GSD}{
\deduce{w:A[z/x],D(\riota xA,z,w),\GSD}{\vdots\;\mathcal{D}_{21}}
}
}
}$$

and we transform it into:

$$
\infer=[\infrule LC+RC]{\g',\GSD,\d'}{
\infer[\infrule Cut_2]{\g',\g',\GSD,\d',\d'}{
\deduce{\g'\To\d', w:A[z/x]}{\vdots\;\mathcal{D}_{11}}
&
\infer[\infrule Cut_1']{w:A[z/x],\g',\GSD,\d'}{
\deduce{\g'\To\d',D(\riota xA,z,w)}{\vdots\;\mathcal{D}_1}&
\deduce{D(\riota xA,z,w),w:A[z/x],\GSD}{\vdots\;\mathcal{D}_{21}}}
}}
$$
where $Cut'_1$ has lesser cut-height and $Cut_2$ has a cut formula of lower weight.

- if the right premise by $LD_2$, we have:\vspace{0.2cm}

$$\scalebox{0.6800}{
\infer[\infrule Cut']{\g',\GSD,\d'}{
\infer[\infrule RD]{\g'\To\d',D(\riota xA,z_1,w)}{
\deduce{\g'\To\d', w:A[z_1/x]}{\vdots\;\mathcal{D}_{11}}&
\deduce{w:A[y/x],\g'\To\d' ,w:z_1=y}{\vdots\;\mathcal{D}_{12}}
}&
\infer[\infrule LD_2]{D(\riota xA,z_1,w),\GSD}{
\deduce{D(\riota xA,z_1,w),\GSD,w:A[z_2/x]}{\vdots\;\mathcal{D}_{21}}
&
\deduce{w: z_1=z_2,D(\riota xA,z_1,w),\GSD}{\vdots\;\mathcal{D}_{22}}
}
}
}$$

\noindent we transform it into:%\vspace{0.2cm}

$$\scalebox{0.800}{
\infer=[\infrule LC+RC]{\g',\GSD,\d'}{\infer[\infrule Cut_2]{\g',\g',\g',\g,\GSD,\d,\d',\d',\d'}{\infer[\infrule Cut'_1]{\g',\GSD,\d', w:A[z_2/x]}{
\deduce{\g'\To\d',D(\riota xA,z_1,w)}{\vdots\;\mathcal{D}_1}&
\deduce{D(\riota xA,z_1,w),\GSD,w:A[z_2/x]}{\vdots\;\mathcal{D}_{21}}
}&
\deduce{w:A[z_2/x],\g',\g',\GSD,\d',\d'}{\vdots\;\mathcal{D}_{3}}
}
}
}$$
%}\vspace{0.2cm}

\noindent where $\D_3$ is the following derivation:

$$\scalebox{0.8000}{
\infer[\infrule Cut_4]{w:A[z_2/x],\g',\g',\GSD,\d',\d'}{
\infer[\infrule{[z_2/y]}]{w:A[z_2/x],\g'\To\d' ,w:z_1=z_2}{\deduce{w:A[y/x],\g'\To\d' ,w:z_1=y}{\vdots\;\mathcal{D}_{12}}}
&
\infer[\infrule Cut'_3]{w:z_1=z_2,\g',\GSD,\d'}{
\deduce{\g'\To\d',D(\riota xA,z_1,w)}{\vdots\;\mathcal{D}_1}
&
\deduce{D(\riota xA,z_1,w),w: z_1=z_2,\GSD}{\vdots\;\mathcal{D}_{22}}
}
}
}$$
 \emph{Cut}$'_1$ and \emph{Cut}$'_3$ are  admissible because they have lesser cut-height; \emph{Cut}$_2$ and \emph{Cut}$_4$ because their cut-formula has lower weight.
\end{proof}

\begin{lemma}[Properties of identity]\label{idprop}\
\begin{enumerate}
\item Identity is an equivalence relation in $\mathbf{G3Q\l.L}$;
\item The following sequents are $\mathbf{G3Q\l.L}$-derivable:
\begin{enumerate}
\item \qquad$\To w: x=x$
\item \qquad$w:z=y,w:A[z/x]\To w:A[y/x]$
\item \qquad $w:z=y,D(t,x_1,w)[z/x]\To D(t,x_1,w)[y/x]$
\end{enumerate}
\item The following rules of replacement are admissible in $\mathbf{G3Q\l.L}$:\\
$$
\infer[\infrule Repl_1]{w:z=y,w:A[z/x],\GSD}{w:A[y/x],w:z=y,w:A[z/x],\GSD}
$$
$$
\infer[\infrule Repl_2]{w:z=y,D(t,x_1,w)[z/x],\GSD}{D(t,x_1,w)[y/x],w:z=y,D(t,x_1,w)[z/x],\GSD}
$$
\end{enumerate}
\end{lemma} 
\begin{proof}
The proofs of items 1 and 2(a) are left to the reader.

We prove cases 2(b) and 2(c) by simultaneous induction on the weight of the `replacement formula'. All cases but that of 2(b) with  $A$  of shape $\l x_1B.t$ and that of 2(c) with $t$ of shape $\riota x_2B$ are left to the reader.

In the first case, assuming, w.l.o.g., $x_1\not\in \{x,y,z\}$, we have:

\noindent$$\scalebox{0.8000}{
\infer[\infrule L\l]{w:z=y,w:\l x_1B[z/x].t[z/x]\To w:\l x_1B[y/x].t[y/x]}{\infer[\infrule R\l]{D(t[z/x],x_2,w),w:B([z/x])[x_2/x_1],w:z=y\To w:\l x_1B[y/x].t[y/x]}{
\infer[\infrule IH_{2(c)}]{D(t[z/x],x_2,w),w:z=y\To D(t[y/x],x_2,w)}{}&
\infer[\star]{w:(B[z/x])[x_2/x_1],w:z=y\To w:(B[y/x])[x_2/x_1]}{\infer[\infrule IH_{2(b)}]{w:(B[x_2/x_1])[z/x],w:z=y\To w:(B[x_2/x_1])[y/x]}{}}}}
}$$
Where the step by  $\star$ is a rewriting that is feasible because $\{x_1,x_2\}\cap\{x,y,z\}=\varnothing$, and the step by $IH_{2(c)/(b)}$ is by induction on case (c)/(b) of the lemma, respectively.

In the second case, assuming $\{x_1,x_2\}\cap\{ x,y,z\}=\varnothing$, we have:

$$
\infer[\infrule RD]{w:z=y,D(\riota x_2(B[z/x]),x_1[z/x],w)\To D(\riota x_2(B[y/x]),x_1[y/x],w)}{
\infer[\infrule LD_1]{w:z=y,D(\riota x_2(B[z/x]),x_1[z/x],w)\To w:(B[y/x])[x_1[y/x]/x_2]}{\infer[\infrule \star]{w:z=y,w:(B[z/x])[x_1[z/x]/x_2]\To w:(B[y/x])[x_1[y/x]/x_2]}{
\infer[\infrule IH_{2(b)}]{w:z=y,w: (B[x_1/x_2])[z/x]\To w: (B[x_1/x_2])[y/x]}{}}}
&
\deduce{\vdots\;\mathcal{D}_2}{}
}
$$

\noindent where the derivation $\mathcal{D}_2$ is as follows:%\vspace{0.2cm}

$$\scalebox{0.74000}{
\infer[\infrule LD_2]{w:B[x_3/x_1[y/x]],w:z=y,D(\riota x_2(B[z/x]),x_1[z/x],w)\To w:x_1[y/x]=x_3}{
\infer[\infrule\star]{w:B[x_3/x_1[y/x]],w:z=y\To w:B[x_3/x_1[z/x]]}{
\infer[\infrule Sym_=]{w:(B[x_3/x_1])[y/x],w:z=y\To w:(B[x_3/x_1])[z/x]}{
\infer[\infrule IH_{2(b)}]{w:(B[x_3/x_1])[y/x],w:y=z\To w:(B[x_3/x_1])[z/x]}{}}}&
\infer[\infrule \star]{w:x_1[z/x]=x_3,w:z=y\To w:x_1[y/x]=x_3}{
\infer[\infrule IH_{2(b)}]{w:(x_1=x_3)[z/x],w:z=y\To w:(x_1=x_3)[y/x]}{}}}
}$$

\noindent Where all the steps marked with $\star$ are syntactic rewritings that are feasible because $\{x_1,x_2,x_3\}\cap\{x,z,y\}=\varnothing$, and the admissibility of \emph{Sym}$_=$ follows from Lemma \ref{idprop}.1.

Finally, item (3) follows from item (2) thanks to the admissibility of \emph{Cut} and \emph{Contraction}, as it is shown by the following derivation for \emph{Repl}$_1$ (the case of \emph{Repl}$_2$ uses \emph{Cut}$'$ and Lemma \ref{idprop}.2(c), and the proof  proceeds in the same way):

$$\scalebox{0.9600}{
\infer[\infrule LC]{w:z=y,w:A[z/x],\GSD}{
\infer[\infrule Cut]{w:z=y,w:z=y,w:A[z/y], w:A[z/x],\GSD}{
\infer[\infrule{\ref{idprop}.2(b)}]{w:z=y,w:A[z/x]\To w:A[y/x] }{}
&
w:A[y/x],w:z=y,w:A[z/x],\GSD
}}
}$$
\end{proof}

%------  DERIVATIONS
\begin{example} If $y$ is new to $z$ and to $A$, then the  $\L^\l$-formula:
$$\l x_1(x_1=z).\riota x A\to(A[z/x]\wedge(A[y/x]\to y=z\;)$$  is {\bf G3Q$\l$.L}-derivable (see \cite[Section 12.5]{FM98} for a discussion of this formula).

$$
\infer[\infrule L\l]{w:\l x_1(x_1=z).\riota x A\To w:A[z/x]}{\infer[\infrule LD_1]{D(\riota xA,y,w),w:y=z\To w:A[z/x]}{
\infer[\infrule Repl_1]{w:A[y/x],D(\riota xA,y,w),w:y=z\To w:A[z/x]}{\infer[\infrule{\text{Lemma }\ref{ax}.1}]{w:A[z/x],w:A[y/x],D(\riota xA,y,w),w:y=z\To w:A[z/x]}{}}}}$$

$$
\infer[\infrule L\l]{w:\l x_1(x_1=z).\riota x A\To w:A[y/x]\to y=z}{
\infer[\infrule R\to]{D(\riota xA,a,w),w:a=z\To w:A[y/x]\to y=z}{
\infer[\infrule LD_2]{w:A[y/x],D(\riota xA,a,w),w:a=z\To w:y=z}{
\infer[\infrule{\ref{ax}.1}]{w:A[y/x],\dots\To w:y=z,w:A[y/x]}{}&
\infer[\infrule{\ref{ax}.1}]{w:y=z,w:A[y/x],\dots\To w:y=z}{}}}}$$
\end{example}

%\begin{proposition}
%Principal cuts are admissible
%\end{proposition}
%
%\begin{proof}\
%
%
%\end{proof}

%----------- CHARACTERIZATION
\section{Soundness and completeness}\label{sec:characterization}
\subsection{Soundness}
\begin{definition}\label{valid}
Given a model $\M=<\W,\R,\D,\V>$, let $f: LAB\cup VAR\longrightarrow \W\cup D_\W$ be a function mapping labels to worlds of the model and mapping variables to objects of the union of the domains of the model. We say that:

\noindent \begin{tabular}{lll}\noalign{\smallskip}
$\M$ satisfies $w:A$ under $f$&iff& $f\models_{f(w)}^\M A$\\\noalign{\smallskip}
$\M$ satisfies $x\in w$ under $f$&iff&$f(x)\in D_{f(w)}$\\\noalign{\smallskip}
$\M$ satisfies $w\RR v$ under $f$&iff& $f(w)\R f(v)$\\\noalign{\smallskip}
$\M$ sat. $D(t,x,w)$ under $f$&iff& $\left \{
\begin{array}{ll}
\forall o\in D_\W(f^{y\tr o}\models^\M_{f(w)}A \textnormal { iff }o=f(x))& \mbox{if $t\equiv \riota yA$}; \\
f(t)=f(x)& \mbox{if $t\equiv y$}; \\
\end{array} \right.  $\\\noalign{\smallskip}

\end{tabular}

\noindent Given a sequent ${}\GSD$ we say that it is \emph{{\bf Q$\l$.L}-valid} iff for  every pair $\M,f$ where $\M$ is a model for {\bf Q$\l$.L}, if $\M$ satisfies under $f$ all formulas in $\g$ then $\M$ satisfies under $f$ some formula in $\d$.
\end{definition}
\begin{theorem}[Soundness]\label{soundness} If a sequent ${}\GSD$ is {\bf G3Q$\l$.L}-derivable, then it is {\bf Q$\l$.L}-valid.
\end{theorem}
\begin{proof}
The proof is by induction on the height of the {\bf G3Q$\l$.L}-derivation of ${}\GSD$. The base case holds since $\g$ and $\d$ have one formula in common, and it is easy to see that the propositional rules, the rules for $\forall$, and the rules for $\Box$ preserve validity on every model, see \cite[Thm. 12.13]{NP11}.

For rule $L\l$, let the last step of $\D$ be:
$$
\infer[\infrule L\l]{{}w:\l x.A.t,\GSD}{D(t,y,w),{}w:A[y/x],\GSD}
$$
with $y$ not free in $w:\l xA.t,\g,\d$. Suppose that $\M$ satisfies under $f$ all formulas in $\g$ and the formula $w:\l x.A.t$. We have to prove that $\M$ satisfies under $f$ also some formula in $\d$. Since $f\models_{f(w)}^\M \l x.A.t$, we know that, in $f(w)$, the term $t$ denotes some object $o\in D_\W$ and that $f^{y\tr o}\models_{f(w)}^\M A[y/x]$, where $y$ does not occur in $\g,A$. This implies that $\M$ satisfies under $f^{y\tr o}$ all formulas in $D(t,y,w),{}w:A[y/x],\g$. Thus, by IH, $\M$ satisfies under $f^{y\tr o}$  some formula in $\d$. Since $y$ does not occur in $\d$, we conclude that $\M$ satisfies under $f$ some formula in $\d$.

For rule $R\l$, let the last step of $\D$ be:
$$
\infer[\infrule R\l]{{}\GSD,w:\l x.A.t}{\GSD,w:\l xA.t,D(t,y,w)\quad&\GSD,w:\l x.A.t,w:A[y/x]}
$$

\noindent We consider an arbitrary pair $\M,f$ satisfying all formulas in $\g$. By IH we know that if $\M,f$ does not   satisfy some formula  in $\d,w:\l x.A.t$, it satisfies both $D(t,y,w)$ and $\l xA[y/x]$. If  $\M,f$ satisfies some formula in $\d,w:\l x.A.t$ there is nothing to prove. Else, $\M$ satisfies under $f$  the formulas $w:A[y/x]$ and $D(t,y,w)$, in this case it is easy to see that $\M$  satisfies under $f$ also $\l x.A.t$.

For rule $LD_1$, let the last step of $\D$ be:
$$
\infer[\infrule LD_1]{D(\riota x_1A,x_2,w),\GSD}{w:A[x_2/x_1],D(\riota x_1A,x_2,w),\GSD}$$

Let us consider a pair $\M,f$ that satisfies all formulas in $\g$ and $D(\riota x_1A,x_2,w)$. The latter means that
\begin{equation}\label{riotafact}
\forall o\in D_\W(f^{x_1\tr o}\models_{f(w)}^\M A\textnormal{ iff } o=f(x_2)\,)
\end{equation}
Hence, we know that $\M,f$ satisfies $w:A[x_2/x_1]$ and, by IH, we conclude that it satisfies also some formula in $\d$.

For rule $LD_2$, we proceed as for $LD_1$. If the last step of $\D$ is:
$$\infer[\infrule LD_2]{D(\riota x_1A,x_2,w)\GSD}{D(\riota x_1A,x_2,w),\GSD,w:A[y/x_1]\quad&w:x_2=y,D(\riota x_1A,x_2,w),\GSD} $$
and if $\M,f$ is a pair that satisfies  $D(\riota x_1A,x_2,w)$, then we know that (\ref{riotafact}) holds. Assume that $\M,f$ satisfies also all formulas in $\g$. By IH, the left premise entails that if that pair does not satisfy some formula in $\d$ then it satisfies $w:A[y/x_1]$ and, hence $f^{x\tr f(y)}\models_{f(w)}^\M A$. But then (\ref{riotafact}) entails that $f(y)=f(x_2)$ -- i.e., $\M,f$ satisfies also $w:x_2=y$ . By induction on the right premiss we conclude that the pair under consideration satisfies some formula in $\d$. %This implies that, for each $z$  $\M,f$ satisfies $w:X_2=y$ if it does not satisfyes

For rule $RD$, let the last step of $\D$ be:
$$
\infer[\infrule{RD,\, z\text{ fresh}}]{\GSD,D(\riota x_1A,x_2,w)}{\GSD,w: A[x_2/x_1]\qquad&w:A[z/x_1],\GSD,w: x_2=z}$$
Let $\M,f$ satisfy all formulas in $\g$. If it also satisfies some formula in $\d$ we are done. Else, by IH, we know it satisfies $w
:A[x_2/x_1]$. Let $o$ be any object such that $f^{x_1\tr o}\models_{f(w)}^\M A$. We consider a pair $\M,f'$ where $\M$ is as before and $f'$ is like $f$ save that $f'(z)=o$. Thanks to the variable condition on $z$, this implies that $\M,f'$ satisfies $w:x_2=z$. Hence, $\M,f'$ satisfies $D(\riota x_1A,x_2,w)$. We conclude the same holds for $\M,f$.

The rules for identity  preserves validity on every model: the proof is standard for rules $Ref_=$ and $Repl$. For rule $RigVar$ it depends on the fact that variables are rigid designators. Also the rules $DenVar$ and $DenId$ preserves validity on every model. For $DenVar$ this depends on the fact that variables denote in every world. For $DenId$, this holds because the fact that  $\M$ satisfies $D(x,y,w)$ under $f$ means that $f(x)=f(y)$. Therefore, $\M$ must also satisfy under $f$ the formula $w:x=y$.

For the proof that each non-logical rule from Table \ref{nonlogicalQK} preserves validity over the appropriate class of frames, we refer the reader to  \cite[Thm. 12.13]{NP11}.%If the last step in $\D$ is by a non-logical rule $R$ of {\bf G3Q$\l$.L}, we can show that $R$ preserves {\bf Q$\l$.L}-validity. To illustrate, if $R$ is ???%$Rig$, we consider a model $\M$ where constants are rigid designators, see Table \ref{propl}. Given a generic  $f$ such that $\M$ satisfies under $f$ all formulas in $w\RR v, x\stackrel{w}\approx t,\o,\g$. We have to prove that $\M$ satisfies under $f$ also some formula in $\d$. By rigidity, $\M$ satisfies under $f$ also $x\stackrel{v}\approx t$ and, by IH, we conclude that it satisfies some formula in $\d$. 
\end{proof}
%--------
\subsection{Completeness}
\begin{theorem}[Completeness]  If a  sequent ${}\GSD$   is {\bf Q$\l$.L}-valid,  it  is {\bf G3Q$\l$.L}-derivable.
\end{theorem}
\begin{proof}
The proof is organized in four main steps. First, in Def.~\ref{saturated}, we sketch a root-first {\bf G3Q$\l$.L}-proof-search procedure. Second, in Def.~\ref{procedure}, we define the notion of saturation for a branch of a {\bf G3Q$\l$.L}-proof-search tree and, in Proposition \ref{konig}, we show that, for every sequent, a {\bf G3Q$\l$.L}-proof-search either gives us a {\bf G3Q$\l$.L}-derivation of that sequent, or it has a saturated branch. Third, in Def.~\ref{mmodel}, we define a model $\M^\B$ out of a saturated branch $\B$. Finally, in Lemma \ref{truth}, we prove that $\M^\B$ is a model for {\bf Q$\l$.L} that falsifies ${}\GSD$. 
\end{proof}
% In this way we'll  obtain a {\em proof-search tree} that can be  either a  derivation, or a non-derivation; the latter can either be a finite search tree that contains finite saturated branches, or an infinite search that, by K\"onig's lemma, contains an infinite saturated branch.  We shall now prove that a saturated branch (either finite or infinite) for a sequent $\GSD$ gives a countermodel.
\begin{definition}\label{saturated} A \emph{{\bf G3Q$\l$.L}-proof-search tree} for a sequent ${}\GSD$ is a tree of sequents generated according to the following inductive procedure. At {\bf step 0} we write the one node tree ${}\GSD$. At {\bf step} $\mathbf{n+1}$, if all leaves  of the tree generated at step $n$ are initial sequents, the procedure ends. Else, we continue the bottom-up construction by applying, to each leaf that is not an initial sequent, each applicable instance of a rule of {\bf G3Q$\l$.L} (by invertibility of the rules, there is no prescribed order in which thee rules  need to be applied) or, if no rule instance is applicable, we copy the leaf on top of itself. For rules \emph{Ref$_=$, Ref$_\W$, Ser, Cons} and \emph{ DenVar}, we consider applicable only instances where, save for \emph{eigenvariables}, all terms and labels  occurring in  the active  formula of that instance already occur in the leaf. See \cite[Thm. 12.14]{NP11} for the details of the inductive procedure (the reader might easily fill  the missing details).
%\begin{itemize}
%\item Step 0:= we write the one-node tree ${}\GSD$;
%\item step  n+1:= if all leaves  of the tree generated at step $n$ are initial sequents the procedure ends. Else, we continue the bottom-up construction by applying, to each leaf that is not an initial sequent, each applicable instance of a rule of {\bf G3Q$\l$.L} as follows:
%\end{itemize}
%\begin{enumerate}
%\item we apply, in parallel, all instances of rules $L\neg$, $R\neg$ and $L\wedge$;
%\item we apply, in some order, all instances of $R\wedge$;
%\item we apply (to the leaves generated at the previous substep), in parallel, all instances of rules $R\forall$, $R\Box$ and $L\l$,  choosing a different fresh variable for each instance;
%\item we apply , in parallel, all instances of rules $L\forall$, $L\Box$ and $R\l$ and all applicable instances of the rules for identity (where $Ref_=$  can be applied only w.r.t. terms occurring in the leaf);
%\item we apply, in parallel, all applicable instances of non-logical rules of {\bf G3Q$\l$.L};
%\item if no rule instance has been applied at substeps 1--5, we copy the leaf on top of itself.
%\end{enumerate}
\end{definition}

\begin{definition}[Saturation]\label{procedure} A branch $\mathcal{B}$ of a {\bf G3Q$\l$.L}-proof-search tree for a sequent  is \emph{{\bf Q$\l$L}-saturated} if it satisfies the following conditions, where $\mathbf{\g}$ ($\mathbf{\d}$) is the union of the antecedents (succedents) occurring in that branch,
\begin{enumerate}
\item no $w:p$ occurs in $\mathbf{\g}\cap\mathbf{\d}$;
\item no $D(y,x,w)$ occurs in $\mathbf{\g}\cap\mathbf{\d}$;
\item  $w:\bot $ does not occur in  $\mathbf{\g}$;
%\item  if $w:\neg A$ is in $ \mathbf{\d}$, then $w:A$ is in $\mathbf{\g}$;
\item  if $w: A\wedge B$ is in $ \mathbf{\g}$, then both $w:A$ and $w:B$ are  in $\mathbf{\g}$;
\item  if $w: A\wedge B$ is in  $\mathbf{\d}$, then at least one of $w:A$ and $w:B$ is  in $\mathbf{\d}$;
\item  if $w: A\lor B$ is in $ \mathbf{\g}$, then at least one of $w:A$ and $w:B$ is  in $\mathbf{\g}$;
\item  if $w: A\lor B$ is in  $\mathbf{\d}$, then both $w:A$ and $w:B$ are  in $\mathbf{\d}$;
\item  if $w: A\to B$ is in $ \mathbf{\g}$, then   $w:A$ is in $\mathbf{\d}$ or   $w:B$ is  in $\mathbf{\g}$;
\item  if $w: A\to B$ is in  $\mathbf{\d}$, then  $w:A$ is in $\mathbf{\g}$ and $w:B$ is  in $\mathbf{\d}$;
\item  if both $w: \fa xA$  and $y\in w$ are in $\mathbf{\g}$, then $w:A[y/x]$ is in $\mathbf{\g}$;
\item  if $w: \fa xA$ is in $ \mathbf{\d}$, then, for some $z$,  $w:A[z/x]$ is in $\mathbf{\d}$ and $z\in w$ is  in $\mathbf{\g}$;
\item  if $w: \ex xA$ is in $ \mathbf{\g}$, then, for some $z$, both  $w:A[z/x]$ and  $z\in w$ are   in $\mathbf{\g}$;
\item  if $w: \ex xA$  is in $\mathbf{\d}$   and $y\in w$ is in  $\mathbf{\g}$, then $w:A[y/x]$ is in $\mathbf{\d}$;
\item   if  both $w: \Box A$  and $w\RR v$ are  in $\mathbf{\g}$, then $v:A $ is in $\mathbf{\g}$;
\item if $w:\Box A$ is in $\mathbf{\d}$, then, for some $u$, $u:A$ is in $\mathbf{\d}$ and $w\RR u$ is in $\mathbf{\g}$;
\item if $w:\Diamond A$ is in $\mathbf{\g}$, then, for some $u$, both $u:A$ and  $w\RR u$ are in $\mathbf{\g}$;
\item   if   $w: \Diamond A$ is in $\mathbf{\d}$    and $w\RR v$ is   in $\mathbf{\g}$, then $v:A $ is in $\mathbf{\d}$;
\item if $w:\l x.A.t$ is in $\mathbf{\g}$, then, for some $z$, both $D(t,z,w)$ and $w:A[z/x]$ are in $\mathbf{\g}$;
\item if $w:\l x.A.t$ is in $\mathbf{\d}$, then, for each $y\,(\in\mathcal{B})$, at least one of $D(t,y,w)$ and $w:A[y/x]$ is in $\mathbf{\d}$;
\item if $D(\riota x_1 A,x_2,w)$ is in $\mathbf{\g}$, then   $w:A[x_2/x_1]$ is in $\mathbf{\g}$ and, for each $y\,(\in\mathcal{B})$, either $w:A[y/x_1]$ is in $\mathbf{\d}$  or $w:x_2=y$ is in $\mathbf{\g}$;
\item if $D(\riota x_1 A,x_2,w)$ is in $\mathbf{\d}$, then $w:A[x_2/x_1]$ is in $\mathbf{\d}$ or, for some $z$, both $w:A[z/x_1]$ is in $\mathbf{\g}$ and $w:x_2=z$ is in $\mathbf{\d}$;
\item if the principal formulas of some instance of one of $Ref_=$,  $Repl$,   $RigVar$, $DenVar$, and $DenId$ is in $\mathbf{\g}$, then also the corresponding active formulas are in $\mathbf{\g}$.

\item[23$_R$.] if $R$ is a non-logical rule of {\bf G3Q$\l$.L}, then for each set of  principal formulas of $R$ that are  in $\mathbf{\g}$ also the corresponding active formulas are in $\mathbf{\g}$ (for some \emph{eigenvariable} of that rule, if any).
\end{enumerate}
\end{definition}

\begin{proposition}\label{konig} Let us consider a {\bf G3Q$\l$.L}-proof-search tree for a sequent $\mathcal{S}$, two cases are possible: either the tree is finite or not. If the tree is finite then all of its leaves are initial sequents and  it grows by applying rules of {\bf G3Q$\l$.L}. Hence,  the tree is a  {\bf G3Q$\l$.L}-derivation of $\mathcal{S}$ and, by Theorem \ref{soundness}, $\mathcal{S}$ is {\bf Q$\l$.L}-vaild. Else, $\mathcal{S}$ is not {\bf G3Q$\l$.L}-derivable and,  by K$\ddot{\textrm{o}}$nig's Lemma, the tree has an infinite branch $\B$ that is {\bf Q$\l$.L}-saturated since every applicable rule instance has been applied at some step of the construction of the tree.
\end{proposition}
\begin{proposition}\label{wellid} Let $\mathbf{\g}$ $(\mathbf{\d}$) be the union of the antecedents (succedents) of a {\bf Q$\l$.L}-saturated branch. It is immediate to notice that, by saturation under rule $Ref_=$ and $Repl$ (cf. Def. \ref{procedure}.22), the set of variables $x,y$ such that $w:x=y$ is in $\mathbf{\g}$ is an equivalence class $[x]_w$. Moreover,  by saturation under  $RigVar$, the same equivalence class holds with respect to each label $v$ occurring in $\mathbf{\g},\mathbf{\d}$ (hence we allow ourselves to use  $[x]$ instead of $[x]_w$).

\end{proposition}
\begin{definition}\label{mmodel} Let $\B$ be a saturated branch of a  {\bf G3Q$\l$.L}-proof-search tree for a sequent, and let $\mathbf{\g}$ be the union of its antecedents. The model $\M^\B=<\W^\B,\R^\B,\D^\B,\V^\B>$ is defined from $\B$ as follows ($\M^\B$ is well-defined thanks to Definition \ref{procedure}.1--3, Lemma \ref{ax}.1--2, and  Proposition \ref{wellid}):
\begin{itemize}
\item $\W^\B$ is the set of all labels occurring in $\B$;
\item $\R^\B$ is such that $w\R^\B v$ iff $w\RR v$ occurs in $\B$;
\item $\D^\B$ is such that, for each $w\in \W^\B$, $D_w$ is the set containing, for each variable $x$ such that $x\in w$ occurs in $\B$, the equivalence class $[x]$ of all  $x,y$ such that $w:x=y$ occurs in $\B$;
\item $\V^\B$ is  such that, for every predicate $P^n\in\mathcal{S}^\l$, $V(P^n,w)$ is the set of all $n$-tuples of equivalence classes of variables $<[x_1],\dots,[x_n]>$ such that  $w:Px_1,\dots,x_n$ occurs in $\mathbf{\g}$.
\end{itemize}
\end{definition}
\begin{lemma}\label{truth} If $\M^\B$ is the model defined from a saturated branch $\B$ of a {\bf G3Q$\l$.L}-proof-search tree for a sequent ${}\GSD$ and $\s$ is the assignment defined by $\s(x)=[x]$, then, for each labelled formula $w:A$ occurring in $\B$ and for each denotation formula $D(t,x,w)$ occurring in $\B$, 
\begin{enumerate}
\item $\s\models_w^{\M^\B} A\qquad\qquad\text{iff}\qquad w:A\textnormal{ occurs in }\mathbf{\g}$
\item $\V^{\B,\s}_w(t)=[x]\;\qquad\textnormal{iff}\qquad D(t,x,w)\textnormal{ occurs in }\mathbf{\g}$ 
\item $\M^\B$ is a model for {\bf Q$\l$.L}.
\end{enumerate}
\end{lemma}
\begin{proof}
The proof of {\bf claims 1 and 2} is by simultaneous   induction on the weight of $w:A$ and of $D(t,x,w)$, respectively. 

We start with {\bf claim 1}. The base case holds thanks to the definition of $\V^\B$, and the inductive cases depends on the construction of $\M^\B$ and on properties 4--19 of the definition of saturated branches \ref{procedure}.
 To illustrate, suppose $w:A\equiv w:\l xB.t$. 
 
 If $w:A$ occurs in $\mathbf{\g}$, then, by Def. \ref{procedure}.18, for some $z$, $D(t,z,w)$ and $w:B[z/x]$ are in $\mathbf{\g}$. By induction on claim 2, this implies  that $\V^{\B,\s}_w(t)=[z]$ and, by induction on claim 1, it also  implies  that $\s^{x\tr[z]}\models_w^{\M^\B} B$. Thus, $\s\models_w^{\M^\B} \l xB.t$.
 
 Suppose that  $w:\l xB.t$ is in $\mathbf{\d}$. If  there is no variable $y$ such that $D(t,y,w)$ is in $\mathbf{\g}$, then, by (the latter fact and by) induction on claim 2,  we immediately have that  $\V^{\B,s}_w(t)$ is undefined. Thus,  $\s\not\models_w^{\M^\B} \l xB.t$. Else,  given Def. \ref{procedure}.2 and Lemma \ref{ax}.2, Def. \ref{procedure}.19  entails that, for each $y$ such that $D(t,y,w)$ is in $\mathbf{\g}$,  $w:B[y/x]$ is in $\mathbf{\d}$. By induction on claim 2, we have that that $\V^{\B,\s}_w(t)=[y]$ and, thanks to induction on claim 1 (and \ref{ax}.1, \ref{procedure}.1),  $\s^{x\tr[y]}\not\models_w^{\M^\B} B$. We conclude that $\s\not\models_w^{\M^\B}\l xB.t$.
 
 Next, we consider {\bf Case 2}. In the base case $t$ is a variable $y$ and, by construction of $\D^\B$, we know that $\V^{\B,\s}_w(y)=[x]$ iff $w:x=y$ occurs in $\mathbf{\g}$. The right-to left implication  holds thanks to saturation under rule \emph{DenId}, and the left-to-right one  thanks to saturation under rules \emph{DenVar} and \emph{Repl} (it is enough to consider an instance of \emph{Repl} with principal formulas $w:x=y$ and $D(z,x,w)[x/z]$).
 
If, instead,  $t\equiv\riota yB$, we make use properties 20 and 21 of the definition of saturated branch to prove that, whenever $D(\riota yB, x,w)$ is in $\mathbf{\g}$, $[x]$ is the only member of $D_{\W^\B}$ such that \mbox{$\s^{y\tr[x]}\models^{\M^\B}_w B$.}
 
 If $D(\riota y B,x,w)$ occurs in $\mathbf{\g}$, then Def. \ref{procedure}.20 entails that (i)  $w:B[x/y]$ is in $\mathbf{\g}$ and  that (ii), for each $z\in\B$, if $w:B[z/y]$ occurs in $\mathbf{\g}$ then also $w:x=z$ occurs in $\mathbf{\g}$. By induction on claim 1 and by construction of $D_{\W^\B}$, fact (i) implies that $[x]$ is such that $\s^{y\tr[x]}\models^{\M^\B}_w B$, and fact (ii) implies that for each $[z]\in D_{\W^\B}$, $\s^{y\tr[z]}\models^{\M^\B}_w $ only if $[z]=[x]$. Thus, we conclude that $\V^{\B,\s}_w(\riota yB)=[x]$.
 
 If $D(\riota yB,x,w)$ occurs in $\mathbf{\d}$, then either $w:B[x/y]$ is in $\mathbf{\d}$ or, for some $z\in\B$, $w:A[z/y]$ is in $\mathbf{\g}$ and $w:x=z$ is in $\mathbf{\d}$. In the first case $[x]$ is such that $\s^{y\tr[x]}\not\models^{\M^\B}_w B$; in the latter there is  $[z]\in D_{\W^\B}$ such that   $\s^{y\tr[z]}\models^{\M^\B}_w B$ and $[z]\neq [x]$. In both cases we can conclude that $\V^{\B,\s}_w(\riota yB)\neq[x]$.

{\bf Claim 3} holds thanks to property $23_R$ of saturated branch: if a non-logical rule $R$ is in {\bf G3Q$\l$.L}, then we have to show that $\M^\B$ satisfies the semantic property corresponding to $R$. This holds by construction of $\M^\B$ since $\B$ is saturated with respect to rule $R$. For example, for rule \emph{Decr}, we have to prove that if $w\R^\B v$ then $D_v\subseteq D_w$.  By construction of $\R^\B$, we know that $w\R^\B v$ implies that $w\RR v$ occurs in $\mathbf{\g}$. Let us now consider a generic $[x]\in \D^\B$. If $[x]\in D_v$ then $x\in v$ occurs in $\mathbf{\g}$. By saturation under rule \emph{Decr}, we have that $x\in w$ and, hence, $[x]\in D_w$. We conclude that $\M^\B$ is based on a frame with decreasing domain.%$Rig$, we have to prove that $w\R^\B v$ and $s_w(t)\in D^\B$ -- i.e., $s_w(t)$ is defined in $\M^\B$ -- implies that $s_v(t)=\s_w(t)$. Suppose that $w\RR v$ is in $\mathbf{\g}$, if $t$ is a variable, then, \ref{procedure}.12 entails that, for all $y$, if $D(t,y,w)$ is in $\mathbf{\g}$, then,  also $y\stackrel{v}\approx t$ is in $\mathbf{\g}$  (by saturation under $DenVar$, $DenId$, $RigVar$, and $Repl$).  Else, $t$ is a constant and if, for some $y$, $D(t,y,w)$ is in $\mathbf{\g}$, saturation under rule $Rig$ entails that $y\stackrel{v}\approx t$ is in $\mathbf{\g}  $. In both cases $\M^\B$ behaves as desired. 
\end{proof}

%------- CONCLUSION
\section{Conclusion}\label{conclusion}
We have introduced labelled sequent calculi that characterize the QMLs with definite descriptions introduced in \cite{FM98}, and we have studied their structural properties. To the best of our knowledge, this is the first proof-theoretic study of these logics. In \cite{FM98} prefixed tableaux for these logics have been considered, but there is no study of their structural properties. Notice that, even if we have considered only the $Q\l$-extensions of propositional modal logics {\bf L}  in the cube of normal modalities, the present approach can be extended, in a modular way, to the $Q\l$-extensions of any propositional modal logic whose class of frames is first-order definable (by applying, if needed, the \emph{geometrisation technique} introduced in \cite{DN15}). For example, we can introduce a calculus characterizing validity in the class of all constant domain frames  satisfying  \emph{confluence}: $$\forall w,v,u\in\W(w\R v\wedge w\R u\to \exists w'\in\W(v\R w'\wedge u\R w'))$$ From \cite{C95}, we know that confluence corresponds to  Geach's axiom  $2:=\Diamond\Box A\to\Box\Diamond A$ and that the quantified modal axiomatic system $Q.2\oplus BF$ is incomplete with respect to the class of all confluent constant domain frames (i.e., the logic {\bf Q$\l$.K2$\oplus$UI}). Nevertheless, confluence is a geometric property, and it can be expressed in labelled calculi by the rule:
$$
\infer[\infrule{ Conf,\, w'\text{ fresh}}]{w\RR v,w\RR u,\GSD}{v\RR w',u\RR w',w\RR v,w\RR u,\GSD}
$$
\noindent It can be proved that the labelled calculus {\bf G3Q$\l$.K}$+\{$\emph{Conf, Cons}$\}$ is sound and complete with respect to the class of confluent constant domain frames.

If, instead, we consider the  logic {\bf Q$\l$.S4.M$\oplus$UI} -- i.e. the set of $\L^\l$-formulas that are valid in the class of constant domain frames that are reflexive, transitive, and \emph{final}:
\begin{equation}\label{final}
\forall w\in \W\exists v\in \W(w\R v\wedge \forall u\in \W(v\R u\to v=u))
\end{equation}
then we have the problem that finality is not a geometric property because of the universal quantifier in the scope of an existential one. Nevertheless, as it is shown in \cite{DN15}, we can transform it into a set of geometric conditions by extending the language with a fresh one place predicate constant $Fin$ that replaces $\forall u\in \W(v\R u\to v=u))$ in (\ref{final}) (thus making it geometric) and that is axiomatized by the following geometric condition:
$$
\forall u\in \W(Fin (v)\wedge v\RR u\to v=u)
$$ 
We can now express finality in labelled calculi by means of the following two geometric rules (where $v=u$ is governed by the rules in \cite[Table 11.7]{NP11}):
$$
\infer[\infrule{ Fin_1,\, v \text{ fresh}}]{\GSD}{w\RR v,Fin (v),\GSD}\qquad
\infer[\infrule Fin_2]{Fin(v), v\RR u,\GSD}{v=u,Fin(v),v\RR u,\GSD}
$$
By applying this geometrisation strategy we can easily obtain a  labelled calculus (with good structural properties) for the {\bf Q}$\l$-extensions of any first-order definable propositional modal logic. This is far more general than any other existing proof-theoretic characterization result for QMLs (see \cite{G18,LR18} for partial translations of these results to internal calculi).
Moreover, given  Proposition \ref{simulate}, we can easily simulate the QMLs with definite descriptions \emph{\`a la} Montague and Kalish or \emph{\`a la} Garson.  All we have to do is to add the following geometric rules for Montague and Kalish's predicate $U$ (rule $U_1$ is dispensable over constant domains):
$$
\infer[\infrule U_1]{w:U(x),\GSD}{x\in w,w:U(x)\GSD}\qquad
\infer[\infrule U_2]{w:U(x),w:U(y),\GSD}{w:x=y,w:U(x),w:U(y),\GSD}
$$
and the following ones for Garson's predicate $U_{\riota xA}$:
$$
\infer[\infrule U_{Ax,1}]{y\in w,w:U_{\riota xA}(y), \GSD}{}\qquad
\infer[\infrule U_{Ax,2}]{w:U_{\riota xA}(z),w:U_{\riota xA}(y), \GSD}{w:z=y,w:U_{\riota xA}(z),w:U_{\riota xA}(y),\GSD}
$$
Hence, we have a labelled version of the Gentzen-style calculi defined in \cite{I18} for Garson's descriptions  and in \cite{I19} for Montague and Kalish's ones. One advantage of the labelled version over the existing Gentzen-style ones is that, whereas the rules of the Getzen-styule calculi  are not apt for proof search, see \cite[Section 4]{I18},   the labelled version allows for proof-search and for the construction of countermodels from a failed proof search. In particular, we bypass the problems noted in \cite[Section 4]{I18} for the rules for identity  because  the language $\L^\l$ is such that  descriptions  cannot occur in identity atoms.
%The calculi introduced here are somehow related to the ones we gave in \cite{CO16} for the \emph{indexed epistemic logics} studied in \cite{CO13}. The QMLs studied here are less general than the ones in \cite{CO13}, but they have the advantage of being simpler and of involving no major departure from the standard modal language. Therefore the present approach to non-rigid and non-denoting terms can easily be extended to most variants of QMLs. For example, it carries over to  labelled calculi for \emph{term-modal logics}  \cite{CO18}.

%---------BIBLIO
\bibliographystyle{plain}
\bibliography{biblio}

\begin{thebibliography}{10}

\bibitem{BI06}
Matthias Baaz and Rosalie Iemhoff.
\newblock Gentzen calculi for the existence predicate.
\newblock {\em Studia Logica}, 82:7--23, 2006.

\bibitem{B09}
Kai Br{\"u}nnler.
\newblock Deep sequent systems for modal logic.
\newblock {\em Archive for Mathematical Logic}, 48(6):551--577, 2009.

\bibitem{G18}
Agata Ciabattoni and Francesco~A. Genco.
\newblock Hypersequents and systems of rules: Embeddings and applications.
\newblock {\em {ACM} Transactions on Computational Logic}, 19(2):11:1--11:27,
  2018.

\bibitem{LR18}
Agata Ciabattoni, Tim Lyon, and Revantha Ramanayake.
\newblock From display to labelled proofs for tense logics.
\newblock In {\em Logical Foundations of Computer Science - International
  Symposium, {LFCS} 2018, Deerfield Beach, FL, USA, January 8-11, 2018,
  Proceedings}, pages 120--139, 2018.

\bibitem{CRW}
Agata Ciabattoni, Revantha Ramanayake, and Heinrich Wansing.
\newblock Hypersequent and display calculi - a unified perspective.
\newblock {\em Studia Logica}, 102(6):1245--1294, 2014.

\bibitem{CO16}
Giovanna Corsi and Eugenio Orlandelli.
\newblock Sequent calculi for indexed epistemic logics.
\newblock In {\em Proceedings of the 2nd International Workshop on Automated
  Reasoning in Quantified Non-Classical Logics {(ARQNL} 2016)}, pages 21--35.
  CEUR-WS.org, 2016.

\bibitem{C95}
Max Cresswell.
\newblock Incompleteness and the barcan formula.
\newblock {\em J. Phil. Logic}, 24(4):379--403, 1995.

\bibitem{DN15}
Roy Dyckhoff and Sara Negri.
\newblock Geometrisation of first-order logic.
\newblock {\em Bulletin of Symbolic Logic}, 21(2):123--163, 2015.

\bibitem{FM98}
Melvin Fitting and Richard~L. Mendelsohn.
\newblock {\em First-Order Modal Logic}.
\newblock Springer, 1998.

\bibitem{F92}
Gottlob Frege.
\newblock {\"U}ber sinn und bedeutung.
\newblock {\em Zeitschrift f\"ur Philosophie Und Philosophische Kritik},
  100(1):25--50, 1892.

\bibitem{G13}
James~W. Garson.
\newblock {\em Modal Logic for Philosophers (2$^{nd}$ ed)}.
\newblock Cambridge University Press, 2013.

\bibitem{I18}
Andrzej Indrzejczak.
\newblock Cut-free modal theory of definite descriptions.
\newblock In {\em Advances in Modal Logic 12}, pages 387--406. College
  Publications, 2018.

\bibitem{I19}
Andrzej Indrzejczak.
\newblock Fregean description theory in proof-theoretical setting.
\newblock {\em Logic and Logical Philosophy}, 28(1):137--155, 2019.

\bibitem{MO19}
Paolo Maffezioli and Eugenio Orlandelli.
\newblock Full cut elimination and interpolation for intuitionistic logic with
  existence predicate.
\newblock {\em Bulletin of the Section of Logic}, 48(2):137--158, 2019.

\bibitem{KM57}
Richard Montague and Donald Kalish.
\newblock Remarks on descriptions and natural deduction.
\newblock {\em Archiv f{\"u}r mathematische Logik und Grundlagenforschung},
  3(1-4):50--73, 1957.

\bibitem{N03}
Sara Negri.
\newblock Contraction-free sequent calculi for geometric theories with an
  application to barr's theorem.
\newblock {\em Archive for Mathematical Logic}, 42(4):389--401, 2003.

\bibitem{NO19}
Sara Negri and Eugenio Orlandelli.
\newblock {Proof theory for quantified monotone modal logics}.
\newblock {\em Logic Journal of the IGPL}, 27(4):478--506, 2019.

\bibitem{NP98}
Sara Negri and Jan von Plato.
\newblock Cut elimination in the presence of axioms.
\newblock {\em Bullettin of Symbolic Logic}, 4(4):418--435, 1998.

\bibitem{NP11}
Sara Negri and Jan von Plato.
\newblock {\em Proof Analysis}.
\newblock Cambridge University Press, 2011.

\bibitem{CO18}
Eugenio Orlandelli and Giovanna Corsi.
\newblock Decidable term-modal logics.
\newblock In Francesco Belardinelli and Estefan{\'i}a Argente, editors, {\em
  Multi-Agent Systems and Agreement Technologies}, pages 147--162. Springer
  International Publishing, 2018.

\bibitem{OC18}
Eugenio Orlandelli and Giovanna Corsi.
\newblock Labelled calculi for quantified modal logics with non-rigid and
  non-denoting terms.
\newblock In {\em Proceedings of the 3rd International Workshop on Automated
  Reasoning in Quantified Non-Classical Logics {(ARQNL} 2018)}, pages 64--78.
  CEUR-WS.org, 2018.

\bibitem{R16}
Revantha Ramanayake.
\newblock From axioms to structural rules, then add the quantifiers.
\newblock In {\em Proceedings of the 2nd International Workshop on Automated
  Reasoning in Quantified Non-Classical Logics {(ARQNL} 2016)}, pages 1--8.
  CEUR-WS.org, 2016.

\bibitem{R05}
Bertrand Russell.
\newblock On denoting.
\newblock {\em Mind}, 14(56):479--493, 1905.

\bibitem{S48}
Arthur~F. Smullyan.
\newblock Modality and description.
\newblock {\em Journal of Symbolic Logic}, 13(1):31--37, 1948.

\bibitem{TS96}
A.~S. Troelstra and H.~Schwichtenberg.
\newblock {\em Basic Proof Theory}.
\newblock Cambridge University Press, New York, NY, USA, 1996.

\end{thebibliography}
\end{paper}
\end{document}